\documentclass[a4paper,10pt]{article}
\usepackage[utf8x]{inputenc}
\usepackage{tracefnt,amsmath,tabu,array}
\usepackage{amssymb,graphicx,setspace,amsfonts,amsbsy}
\usepackage{pifont,latexsym,ifthen,amsthm,rotating,calc,textcase,booktabs}

\newtheorem{theorem}{Theorem}[section]
\newtheorem{lemma}[theorem]{Lemma}
\newtheorem{corollary}[theorem]{Corollary}
\newtheorem{proposition}[theorem]{Proposition}
\newtheorem{statement}[theorem]{Statement}
\newtheorem{definition}[theorem]{Definition}

\newtheorem{remark}[theorem]{Remark}
\newcommand{\filledbox}{\leavevmode
  \hbox to.77778em{%
  \hfil\vbox to.675em{\hrule width.6em height.6em}\hfil}}

\newcommand{\Rm}{{\mathbb R}}
\newcommand{\Hm}{{\mathbb H}}
\newcommand{\Db}{{\mathbf D}}
\newcommand{\Di}{{\mathbf T}}
\newcommand{\eps}{\varepsilon}

\title{A Semi-linear Energy Critical Wave Equation With Applications}
\author{Ruipeng Shen\\
Department of Mathematics and Statistics\\
McMaster University\\
Hamilton, ON, Canada}

\begin{document}
\maketitle

\begin{abstract}
In this work we consider a semi-linear energy critical wave equation in $\Rm^d$ ($3\leq d \leq 5$)
\[
 \partial_t^2 u - \Delta u =  \pm \phi(x) |u|^{4/(d-2)} u, \qquad (x,t)\in \Rm^d \times \Rm
\]
with initial data $(u, \partial_t u)|_{t=0} = (u_0,u_1) \in \dot{H}^1 \times L^2 (\Rm^d)$. Here the function $\phi \in C(\Rm^d; (0,1])$ converges to zero as $|x| \rightarrow \infty$. We follow the same compactness-rigidity argument as Kenig and Merle applied in their paper \cite{kenig} on the Cauchy problem of the equation 
\[ 
 \partial_t^2 u - \Delta u = |u|^{4/(d-2)} u
\]
and obtain a similar result when $\phi$ satisfies some technical conditions. In the defocusing case we prove that the solution scatters for any initial data in the energy space $\dot{H}^1 \times L^2$. While in the focusing case we can determine the global behaviour of the solutions, either scattering or finite-time blow-up, according to their initial data when the energy is smaller than a certain threshold. 
\end{abstract}

\section{Introduction}

In this work we consider a semi-linear energy critical wave equation in $\Rm^d$ with $3\leq d \leq 5$:
\[
 \left\{\begin{array}{ll} \displaystyle \partial_t^2 u - \Delta u = \zeta \phi(x) |u|^{p_c -1} u, & (x,t)\in \Rm^d \times \Rm;\\
  u(\cdot ,0) = u_0  \in \dot{H}^1 (\Rm^d); &\\
  \partial_t u(\cdot,0) = u_1 \in L^2 (\Rm^d); & 
 \end{array}\right. \qquad \hbox{(CP1)}
\]
Here the coefficient function $\phi(x)$ satisfies 
\begin{align} \label{basic condition of phi}
 &\phi \in C(\Rm^d; (0,1]),& &\lim_{|x| \rightarrow \infty} \phi (x) = 0.&
\end{align}
The exponent $p_c = 1+ \frac{4}{d-2}$ is energy-critical and $\zeta = \pm 1$. If $\zeta = 1$, then the equation is called focusing, otherwise defocusing. Solutions to this equation satisfy an energy conservation law:
\begin{equation} \label{energy of CP1}
 E_\phi (u, \partial_t u) = \int_{\Rm^d} \left(\frac{1}{2} |\nabla u|^2 + \frac{1}{2} |\partial_t u|^2 - \frac{\zeta}{2^\ast} \phi |u|^{2^\ast} \right) dx = E_\phi (u_0,u_1). 
\end{equation}
Here the notation $2^\ast$ represents the constant 
\[
 2^\ast = \frac{2d}{d-2}. 
\]
By the Sobolev embedding $\dot{H}^1 (\Rm^d) \hookrightarrow L^{2^\ast} (\Rm^d)$, the energy $E_\phi (u_0,u_1)$ is finite for any initial data $(u_0,u_1) \in \dot{H}^1 \times L^2 (\Rm^d)$. 

\subsection{Background} 

\paragraph{Pure Power-type Nonlinearity} Wave equations with a similar nonlinearity have been extensively studied in many works over a few decades, in particular with a power-type nonlinearity $\zeta |u|^{p-1} u$. There is a large group of symmetries acts on the set of solutions to an equation of this kind. For example, if $u(x,t)$ is a solution to 
\begin{equation} \label{NLW with exponent p}
 \partial_t^2 u - \Delta u = \zeta |u|^{p-1} u 
\end{equation}
with initial data $(u_0,u_1)$, then $\displaystyle \tilde{u} (x,t) \doteq  \frac{1}{\lambda^{\frac{2}{p-1}}} u\left(\frac{x-x_0}{\lambda}, \frac{t-t_0}{\lambda}\right)$ is another solution to \eqref{NLW with exponent p} with initial data
\[
 \left(\frac{1}{\lambda^{\frac{2}{p-1}}} u_0 \left(\frac{x-x_0}{\lambda}\right), \frac{1}{\lambda^{\frac{2}{p-1}+1}} u_1 \left(\frac{x-x_0}{\lambda}\right)\right)
\]
at $t = t_0$, where $\lambda > 0$, $x_0 \in \Rm^d$ and $t_0 \in \Rm$ are arbitrary constants. One can check that the energy defined by
\[
 E(u, \partial_t u) = \int_{\Rm^d} \left(\frac{1}{2} |\nabla u|^2 + \frac{1}{2} |\partial_t u|^2 - \frac{\zeta}{p+1} |u|^{p+1} \right) dx
\]
is preserved under the transformations defined above, i.e. $E(u, \partial_t u) = E (\tilde{u}, \partial_t \tilde{u})$, if and only if $p = p_c \doteq 1 + \frac{4}{d-2}$. This is the reason why the exponent $p_c$ is called the energy-critical exponent, and why the equation \eqref{NLW with exponent p} with $p = p_c$ is called an energy-critical nonlinear wave equation. 

\paragraph{Previous Results} A large number of papers have been devoted to the study of wave equations with a power-type nonlinearity. For instance almost complete results about Strichartz estimates, which is the basis of a local theory, can be found in \cite{strichartz, endpointStrichartz}. Local and global well-posedness has been considered for example in \cite{loc1, ls}. In particular, there are a lot of works regarding the global existence and well-posedness of solutions with small initial data such as \cite{smallgs2, smallgs3, gwpwrn, smallgs1}. Questions on global behaviour of larger solutions, such as scattering and blow-up, are usually considered more subtle.  Grillakis \cite{mg1, mg2} and Shatah-Struwe \cite{ss1,ss2} proved the global existence and scattering of solutions with any $\dot{H}^1 \times L^2$ initial data in the energy-critical, defocusing case  in 1990's. The focusing, energy-critical case has been the subject of several more recent papers. This current work is motivated by one of them, F. Merle and C. Kenig's work \cite{kenig}. I would like to describe briefly its main results and ideas  here. 

\paragraph{Merle and Kenig's work} Let us consider the focusing, energy-critical wave equation 
\[
 \left\{\begin{array}{ll} \displaystyle \partial_t^2 u - \Delta u = |u|^{p_c -1} u, & (x,t)\in \Rm^d \times \Rm;\\
  u(\cdot ,0) = u_0  \in \dot{H}^1 (\Rm^d); &\\
  \partial_t u(\cdot,0) = u_1 \in L^2 (\Rm^d); & 
 \end{array}\right. \qquad \hbox{(CP0)}
\]
Unlike the defocusing case, the solutions to this equation do not necessarily scatter. The ground states, defined as the solutions of (CP0) independent of the time $t$ and thus solving the elliptic equation $-\Delta W = |W|^{p_c-1} W$, are among the most important counterexamples. One specific example of the ground states is given by the formula
\[
 W(x,t) = W(x) = \frac{1}{\left(1 + \frac{|x|^2}{d(d-2)}\right)^{\frac{d-2}{2}}}.
\] 
Kenig and Merle's work classifies all solutions to (CP0) whose energy satisfies the inequality  
\[
 E(u_0,u_1) \doteq \int_{\Rm^d} \left(\frac{1}{2} |\nabla u_0|^2 + \frac{1}{2} |u_1|^2 - \frac{1}{2^\ast} |u_0|^{2^\ast} \right) dx < E(W,0)
\]
into two categories:
\begin{itemize}
 \item[(I)] If $\|\nabla u_0\|_{L^2} < \|\nabla W\|_{L^2}$, then the solution $u$ exists globally in time and scatters. The exact meaning of  scattering is explained in Definition \ref{def of scattering} blow. 
 \item[(II)] If $\|\nabla u_0\|_{L^2} > \|\nabla W\|_{L^2}$, then the solution blows up within finite time in both two time directions. 
\end{itemize}
Please note that $\|\nabla u_0\|_{L^2} = \|\nabla W\|_{L^2}$ can never happen if $E(u_0,u_1) < E(W,0)$. Thus the classification is complete under the assumption that $E(u_0,u_1) < E(W,0)$. The scattering part of this result is proved via a compactness-rigidity argument, which consists of two major steps.
\begin{itemize}
 \item[(I)] If the scattering result were false, then there would exist a non-scattering solution to (CP0), called a ``critical element'', with a minimal energy among those non-scattering solutions, that has a compactness property up to dilations and space translations. 
 \item[(II)] A ``critical element'' as described above does not exist. 
\end{itemize}

\paragraph{Solutions with a greater energy} Before introducing the main results,  the author would like to mention a few works that discuss the properties of the solutions to (CP0) with an energy $E \geq E(W,0)$. These works include \cite{secret, tkm1, radial dynamics} (Radial case) and \cite{nonradial dynamics} (Non-radial Case).

\subsection{Main Results of this work} 

In this work, we will prove that similar results as mentioned in the previous subsection still hold for the equation (CP1), at least for those $\phi$'s that satisfy some additional condition besides \eqref{basic condition of phi}. 

\paragraph{Defocusing Case}  As in the case of the wave equation with a pure power-type nonlinearity, we expect that all solutions in the defocusing case scatter. In fact we have 

\begin{theorem} \label{main theorem defocusing}
Let $3 \leq d\leq 5$. Assume the coefficient function $\phi \in C^1 (\Rm^d)$ satisfies the condition \eqref{basic condition of phi} and 
\begin{equation} \label{condition phi defocusing}
 \phi(x)- \frac{(d-2) x\cdot \nabla \phi(x)}{2(d-1)} > 0,\qquad \hbox{for any}\quad x \in \Rm^d.
\end{equation}
Then the solution to the Cauchy Problem (CP1) in the defocusing case with any initial data $(u_0,u_1) \in \dot{H}^1 \times L^2 (\Rm^d)$ exists globally in time and scatters. 
\end{theorem}

\begin{remark}
Any positive radial $C^1$ function satisfies the condition \eqref{condition phi defocusing} as long as it decreases as the radius $r = |x|$ grows .
\end{remark}

\paragraph{Focusing Case} As in the case of a pure power-type nonlinearity, we can classify all solutions with an energy smaller than a certain positive constant. The threshold here is again the energy of the ground state $W$ for the equation (CP0), defined by
\[
 E_1 (W,0) = \int_{\Rm^d} \left(\frac{1}{2} |\nabla W|^2 - \frac{1}{2^\ast} |W|^{2^\ast} \right) dx.
\]
Please note that $W$ is no longer a ground state of (CP1) and that the energy above is not the energy $E_\phi (W,0)$ for the equation (CP1) as defined in \eqref{energy of CP1}. This can be explained by the following fact\footnote{Without loss of generality, we assume the value of $\phi$ is equal to the upper bound $1$ somewhere; otherwise the threshold can be improved, see corollary \ref{critical focusing phi equal c}.}: If $\phi(x_0) = 1$, then the rescaled version of $W$ defined by
\[
 W_{\lambda, x_0} (x) = \frac{1}{\lambda^{\frac{d-2}{2}}} W \left(\frac{x-x_0}{\lambda}\right)
\]
is ``almost'' a ground state for (CP1) as $\lambda \rightarrow 0^+$ with its energy $E_\phi (W_{\lambda, x_0}, 0) \rightarrow E_1 (W,0)$, as shown by Lemma \ref{approximation solution 1}.  

\begin{theorem} \label{main theorem focusing}
Let $3 \leq d\leq 5$. Assume the function $\phi \in C^1 (\Rm^d)$ satisfies the condition \eqref{basic condition of phi} and 
\begin{equation} \label{condition phi focusing}
 2^\ast (1 - \phi(x)) + (x\cdot \nabla \phi(x)) \geq 0,\qquad \hbox{for any}\quad x \in \Rm^d.
\end{equation} 
Given initial data $(u_0,u_1) \in \dot{H}^1 \times L^2 (\Rm^d)$ with an energy $E_\phi (u_0,u_1) < E_1 (W,0)$, the global behaviour, and in particular, the maximal interval of existence $I = (- T_- (u_0,u_1), T_+ (u_0,u_1))$ of the corresponding solution $u$ to the Cauchy problem (CP1) in the focusing case can be determined by:
\begin{itemize}
 \item[(i)] If $\|\nabla u_0\|_{L^2} < \|\nabla W\|_{L^2}$, then $I = \Rm$ and $u$ scatters in both time directions. 
 \item[(ii)] If $\|\nabla u_0\|_{L^2} > \|\nabla W\|_{L^2}$, then $u$ blows up within finite time in both two directions, namely 
 \begin{align*}
  &T_- (u_0,u_1) < +\infty;& &T_+ (u_0,u_1) < +\infty.&
 \end{align*}
\end{itemize}
\end{theorem}

\begin{remark} \label{example of phi focusing}
 The function $\phi (x) = (\frac{|x|}{\sinh |x|})^\sigma$ satisfies the conditions in Theorem \ref{main theorem focusing} as long as $2 \leq \sigma \leq 2^\ast$. 
\end{remark}

\begin{remark}
The compactness process works for any $\phi$ that satisfies the basic assumption \eqref{basic condition of phi}. Thus the main theorem might still work without the assumption \eqref{condition phi defocusing} or \eqref{condition phi focusing}, if we could develop a successful rigidity theory for more general $\phi$'s. 
\end{remark}

\subsection{Idea of the proof}

In this subsection we briefly describe the idea for the scattering part of our main theorems. We focus on the focusing case, but the defocusing case, that is less difficult, can be handled in the same way. Let us first introduce ($M>0$)

\begin{statement}[SC($\phi$, M)]
 There exists a function $\beta: [0,M) \rightarrow \Rm^+$, such that if the initial data $(u_0,u_1) \in \dot{H}^{1} \times L^2 (\Rm^d)$ satisfy
 \begin{align*}
  &\|\nabla u_0\|_{L^2} < \|\nabla W\|_{L^2},& &E_\phi (u_0,u_1) < M;&
 \end{align*}
 then the solution $u$ to (CP1) in the focusing case with the initial data $(u_0,u_1)$ exists globally in time, scatters in both two time directions with 
 \[
  \|u\|_{L^{\frac{d+2}{d-2}} L^{\frac{2(d+2)}{d-2}}(\Rm \times \Rm^d)} < \beta (E_\phi (u_0,u_1)).
 \]
\end{statement}

\begin{remark}
According to Remark \ref{positive energy}, if $\|\nabla u_0\|_{L^2} < \|\nabla W\|_{L^2}$, then we have
\[
 E_\phi (u_0,u_1) \simeq \|(u_0,u_1)\|_{\dot{H}^1 \times L^2}^2 \geq 0.
\]
Therefore we have 
\begin{itemize}
 \item The expression $\beta (E_\phi (u_0,u_1))$ is always meaningful. 
 \item Proposition \ref{scattering with small data} guarantees that the statement SC($\phi$, M) is always true if $M > 0$ is sufficiently small.
\end{itemize}
\end{remark}

\paragraph{Compactness Process} It is clear that the statement SC($\phi$, $E_1(W,0)$) implies the scattering part of our main theorem \ref{main theorem focusing}. If the statement above broke down at $M_0 < E_1 (W,0)$, i.e. SC($\phi$, M) holds for $M = M_0$ but fails for any $M > M_0$, then we would find a sequence of non-scattering solutions $u_n$'s with initial data $(u_{0,n}, u_{1,n})$, such that $E_\phi (u_{0,n}, u_{1,n})\rightarrow M_0$. In this case a critical element can be extracted as the limit of some subsequence of $\{u_n\}$ by applying the profile decomposition. This process is somewhat standard for the wave or Schr\"{o}dinger equations. However, this is still some difference between our argument and that for a wave equation with a pure power-type nonlinearity. The point is that dilations and space translations are no longer contained in the symmetric group of this equation. The situation is similar when people are considering the compactness process for wave/Schr\"{o}dinger equations on a space other than the Euclidean spaces, see \cite{profile decom1, profile decom2}, for instance. We start by introducing the profile decomposition, before more details are discussed. 

\paragraph{The profile decomposition} One of the key components in the compactness process is the profile decomposition. Given a sequence $(u_{0,n}, u_{1,n}) \in \dot{H}^1 \times L^2 (\Rm^d)$, we can always find a subsequence of it, still denoted by $\{(u_{0,n}, u_{1,n})\}_{n \in {\mathbb Z}^+}$, a sequence of free waves (solutions to the linear wave equation), denoted by $\{V_j (x,t)\}_{j \in {\mathbb Z}^+}$, and a triple $(\lambda_{j,n}, x_{j,n}, t_{j,n}) \in \Rm^+ \times \Rm^d \times \Rm$ for each pair $(j,n)$, such that 
\begin{itemize}
 \item For each integer $J>0$, we have the decomposition
\[
 (u_{0,n}, u_{1,n}) = \sum_{j=1}^J \left(V_{j,n}(\cdot, 0), \partial_t V_{j,n}(\cdot, 0)\right) + (w_{0,n}^J, w_{0,n}^J).
\]
Here $V_{j,n}$ is a modified version of $V_j$ via the application of a dilation, a space translation and/or a time translation:
\[
 \left(V_{j,n}(x, t), \partial_t V_{j,n}(x, t) \right)= \left(\frac{1}{\lambda^{\frac{d-2}{2}}} V_j \left(\frac{x-x_{j,n}}{\lambda_{j,n}}, \frac{t-t_{j,n}}{\lambda_{j,n}}\right), \frac{1}{\lambda^{\frac{d}{2}}} \partial_t V_j \left(\frac{x-x_{j,n}}{\lambda_{j,n}}, \frac{t-t_{j,n}}{\lambda_{j,n}}\right)\right);
\]
and $(w_{0,n}^J, w_{1,n}^J)$ represents a remainder that gradually becomes negligible as $J$ and $n$ grow.
 \item The sequences $\{(\lambda_{j,n}, x_{j,n}, t_{j,n})\}_{n \in {\mathbb Z}^+}$ and $\{(\lambda_{j',n}, x_{j',n}, t_{j',n})\}_{n \in {\mathbb Z}^+}$ are ``almost orthogonal'' for $j\neq j'$. More precisely we have
\[
 \lim_{n\rightarrow \infty} \left(\frac{\lambda_{j,n}}{\lambda_{j',n}} + \frac{\lambda_{j',n}}{\lambda_{j,n}} + \frac{|x_{j,n}-x_{j',n}|}{\lambda_{j,n}} + \frac{|t_{j,n}-t_{j',n}|}{\lambda_{j,n}}\right) = +\infty.
\]
 \item We can also assume $\lambda_{j,n} \rightarrow \lambda_j \in [0,\infty) \cup \{\infty\}$, $x_{j,n} \rightarrow x_j \in \Rm^d \cup \{\infty\}$ and $-t_{j,n}/\lambda_{j,n} \rightarrow t_j \in \Rm \cup \{\infty, -\infty\}$ as $n \rightarrow \infty$ for each fixed $j$. 
\end{itemize}
\paragraph{The nonlinear profile} In the case of a pure power-type nonlinearity, we can approximate the solution to (CP0) with initial data $\left(V_{j,n}(\cdot, 0), \partial_t V_{j,n}(\cdot, 0)\right)$ by a nonlinear profile $U_j$, which is another solution to (CP0), up to a dilation, a space translation and/or a time translation, and then add these approximations up to obtain an approximation of $u_n$, thanks to the almost orthogonality. The fact that the equation (CP0) is invariant under dilations and space/time translations plays a crucial role in this argument. As a result, this can no longer be done for the equation (CP1). However, this problem can still be solved if we allow the use of nonlinear profiles that are not necessarily solutions to (CP1) but possibly solutions to other related equations instead. In fact, the solution to (CP1) with initial data $\left(V_{j,n}(\cdot, 0), \partial_t V_{j,n}(\cdot, 0)\right)$ can be approximated by a nonlinear profile $U_j$ as described below, up to a dilation, a space translation and/or a time translation.
\begin{itemize}
 \item[I] {(Expanding Profile)} If $\lambda_j= \infty$, then the profile spreads out in the space as $n \rightarrow \infty$. Eventually a given compact set won't contain any significant part of the profile. The combination of this fact and our assumption $\lim_{|x| \rightarrow \infty} \phi (x) = 0$ implies that the nonlinear term $\phi(x) |u|^{p_c -1} u$ is actually negligible as $n \rightarrow \infty$. As a result, the nonlinear profile $U_j$ in this case is simply  a solution to the free linear wave equation. 
 \item[II] {(Traveling Profile)} If $\lambda_j < \infty$ but $x_j = \infty$, then the profile travels to the infinity as $n \rightarrow \infty$. Again this enables us to ignore the nonlinear term and choose a nonlinear profile from the solutions to the linear wave equation. In fact, we can make the remainder term absorb those profiles in the cases (I) and (II). 
 \item[III] {(Stable Profile)} If $\lambda_j \in \Rm^+$ and $x_j \in \Rm^d$, then the profile approaches a limiting scale and position as $n \rightarrow \infty$. Therefore the nonlinear profile $U_j$ is a solution to (CP1).
 \item[IV] {(Concentrating Profile)} If $\lambda_j = 0$ and $x_j \in \Rm^d$, then the profile concentrates around a fixed point $x_j$ as $n \rightarrow \infty$. The nonlinear term $\phi(x) |u|^{p_c -1} u$ performs almost the same as $\phi(x_j) |u|^{p_c -1} u$. As a result, the nonlinear profile $U_j$ is a solution to $\partial_t^2 u - \Delta u = \phi (x_j) |u|^{p_c-1} u$.
\end{itemize}

\paragraph{Extraction of a critical element} After the nonlinear profiles are assigned, we can proceed step by step
\begin{itemize}
 \item[(I)] First of all, we show there is at least one non-scattering profile $U$, whose energy is at least $M_0$.
 \item[(II)] By considering the estimates regarding the energy, we show that $U$ is the only nonzero profile and its energy is exactly $M_0$. This also implies that this nonlinear profile is a solution to (CP1). 
 \item[(III)] Finally we prove that the solution $U$ is ``almost periodic'', i.e. the set 
 \[
  \{(U(\cdot, t), \partial_t U(\cdot,t))| t \in I\}
 \]
 is pre-compact in $\dot{H}^1 \times L^2 (\Rm^d)$, where $I$ is the maximal lifespan of $U$, by considering a new sequence of solutions derived from $U$ via time translations and repeating the whole compactness process. A direct corollary is that the maximal lifespan $I$ is actually $\Rm$.
\end{itemize}

\paragraph{Nonexistence of a critical element} Finally we show that a critical element may never exist. 
\begin{itemize}
 \item In the defocusing case, we apply a Morawetz-type inequality, which gives a global integral estimate. This contradicts with the ``almost periodicity''. 
 \item In the focusing case, we follow the same idea used in Kenig and Merle's work. We show that the derivative
\[
 \frac{d}{dt} \left[\int_{\Rm^d}(x \cdot \nabla u) u_t \varphi_R dx + \frac{d}{2} \int_{\Rm^d} \varphi_R u u_t dx\right]
\]
has a negative upper bound but the integral itself is always bounded for all time $t$. This gives us a contradiction when we consider a long time interval. Here $\varphi_R$ is a cut-off function. 
\end{itemize}

\subsection{Structure of this Paper}

This paper is organized as follows: In Section 2 We make a brief review on some preliminary results such as the Strichartz estimates, the local theory and some results regarding the wave equation with a pure power-type nonlinearity. We then consider the linear profile decomposition, define the nonlinear profiles and discuss their properties in Section 3. After finishing the preparation work, we perform the crucial compactness procedure and extract a critical element in Section 4. Next we prove that the critical element can never exist, thus finish the proof of the scattering part of our main theorem in Section 5.  Finally in Sections 6 we prove the blow-up part of our main theorem. Section 7 is an extra, showing an application of our main theorem, about the radial solutions to the focusing, energy-critical shifted wave equation on the 3-dimensional hyperbolic space. 

\section{Preliminary Results}

\subsection{Notations}

\begin{definition} 
 Throughout this paper the notation $F$ represents the function $F(u) = \zeta |u|^{p_c-1} u$. The parameter $\zeta = \pm 1$ is determined by whether the equation in question is focusing ($\zeta =1$) or defocusing ($\zeta = -1$). 
\end{definition}

\begin{definition}[Dilation-translation Operators]
 We define $\Di_\lambda$ to be the dilation operator
 \[
  \Di_\lambda \left(u_0(x) ,u_1(x) \right) = \left(\frac{1}{\lambda^{d/2 - 1}} u_0\left(\frac{x}{\lambda}\right),
\frac{1}{\lambda^{d/2}} u_1\left(\frac{x}{\lambda}\right)\right);
 \]
 and $\Di_{\lambda, x_0}$ to be the dilation-translation operator 
 \[
  \Di_{\lambda,x_0} \left(u_0(x) ,u_1(x) \right) = \left(\frac{1}{\lambda^{d/2 -1}} u_0 \left(\frac{x-x_0}{\lambda}\right),
\frac{1}{\lambda^{d/2}} u_1\left(\frac{x-x_0}{\lambda}\right)\right);
 \]
Here $x$ is the spatial variable of the functions. Similarly we can define these operators in the same manner when both the input and output are written as column vectors. 
\end{definition}

\begin{definition}
 Let $S_L (t)$ be the linear propagation operator, namely we define
\begin{align*}
 &S_L (t_0) (u_0,u_1) = u(t_0)& &S_L (t_0) \begin{pmatrix} u_0\\ u_1
 \end{pmatrix} = 
 \begin{pmatrix} u(t_0)\\ u_t (t_0)
 \end{pmatrix}&
\end{align*}
if $u$ is the solution to the linear wave equation 
\[
 \left\{\begin{array}{l}
  \partial_t^2 u - \Delta u = 0; \\
  (u, \partial_t u)|_{t=0} = (u_0,u_1).
 \end{array}\right.
\]
Similarly $S_{1} (t)$ and $S_{\phi} (t)$ represents the nonlinear propagation operator with nonlinearity $F(u)$ and $\phi F(u)$, respectively. 
\end{definition}

\begin{definition}[The energy] \label{definition of energy with general phi}
Let $\phi: \Rm^d \rightarrow [0,1]$ be a function and $\zeta = \pm 1$. We define $E_{\zeta, \phi} (u_0,u_1)$ to be the energy of the nonlinear wave equation $\partial_t^2 u - \Delta u =  \zeta \phi |u|^{p_c-1} u$ with initial data $(u_0,u_1) \in \dot{H}^1 \times L^2 (\Rm^d)$. In other words, we define 
\[
 E_{\zeta, \phi} (u_0,u_1) = \int_{\Rm^d} \left(\frac{1}{2} |\nabla u_0|^2 + \frac{1}{2} |u_1|^2 - \frac{\zeta}{2^\ast} \phi |u_0|^{2^\ast} \right) dx. 
\]
We may omit $\zeta$ and use $E_\phi (u_0,u_1)$ instead when the choice of $\zeta$ is obvious. 
\end{definition}

\begin{definition} [Space-time Norms]
Assume $1 \leq q, r \leq \infty$ and let $I$ be a time interval. The norm $\|u\|_{L^q L^r (I \times \Rm^d)}$ represents $\left\|\|u(x,t)\|_{L^r (\Rm^d; dx)}\right\|_{L^q (I; dt)}$. In particular, if $1 \leq q,r < \infty$, then we have 
\[
 \|u\|_{L^q L^r (I \times \Rm^d)} = \left(\int_{I} \left(\int_{\Rm^d} |u(x,t)|^r dx \right)^{q/r} dt\right)^{1/q}
\]
\end{definition}

\begin{definition} [Function Spaces] Let $I$ be a time interval. We define the following norms. 
\begin{align*}
 &\|u\|_{Y(I)} = \|u\|_{L^{p_c} L^{2 p_c} (I \times \Rm^d)}& &\|u\|_{Y^\star (I)} = \|\phi^{1/p_c}\|_{Y(I)} = \|\phi^{1/p_c} u\|_{L^{p_c} L^{2p_c} (I \times \Rm^d)}&\\
 &\|(u_0,u_1)\|_H = \|(u_0,u_1)\|_{\dot{H}^1 \times L^2 (\Rm^d)}&
\end{align*}
\end{definition}

\subsection{Local Theory}

In this subsection we briefly discuss the local theory of the nonlinear equation (CP1). Our local theory is based on the Strichartz estimates below. 

\begin{proposition}[Strichartz estimates, see \cite{strichartz}] \label{Strichartz estimates}
There is a constant $C$ determined solely by the dimension $d \in \{3,4,5\}$, such that if $u$ is a solution to the linear wave equation 
 \[
  \left\{\begin{array}{ll} \partial_t^2 u - \Delta u = F(x,t); & (x,t) \in \Rm^d \times I;\\
  (u,\partial_t u)|_{t=0} = (u_0,u_1);&
  \end{array}\right.
 \]
 where $I$ is a time interval containing $0$; then we have the inequality 
 \[
  \sup_{t \in I} \|(u(\cdot,t), \partial_t u(\cdot,t))\|_{\dot{H}^1 \times L^2 (\Rm^d)} + \|u\|_{Y(I)} \leq C \left[ \|(u_0, u_1)\|_{\dot{H}^1 \times L^2 (\Rm^d)} + \|F\|_{L^1 L^2 (I \times \Rm^d)} \right].
 \]
\end{proposition}

\begin{remark}
 We can substitute $Y$ norm in Proposition \ref{Strichartz estimates} by any $L^q L^r$ norm if 
\begin{align*}
 &2 \leq q \leq \infty,& &2 \leq r < \infty,& &\frac{1}{q} + \frac{d}{r} = \frac{d}{2} - 1,&
\end{align*}
as shown in the paper \cite{strichartz}. These space-time norms are called (energy-critical) Strichartz norms. 
\end{remark}

\begin{definition} [Solutions]
Let $(u_0,u_1)$ be initial data in $\dot{H}^1 \times L^2 (\Rm^d)$ and $I$ be a time interval containing $0$.  We say $u(t)$ is a solution of (CP1) defined on the time interval $I$, if $(u(t),\partial_t u(t)) \in C(I;{\dot{H}^{1}}\times L^2 (\Rm^d))$ comes with a finite norm $\|u\|_{Y^\star (J)}$ for any bounded closed interval $J \subseteq I$ and satisfies the integral equation
\[
 u(t) = S(t)(u_0,u_1) + \int_0^t \frac{\sin ((t-\tau)\sqrt{-\Delta})}{\sqrt{-\Delta}} \left[\phi F(u(\cdot, \tau))\right] d\tau
\]
holds for all time $t \in I$. 
\end{definition}

\begin{remark} \label{Y Ystar equivalent}
We can substitute $Y^\star (I)$ norm above with $Y(I)$ norm to make an equivalent definition, by applying the Strichartz estimate on the integral equation above  
\begin{align*}
 \|u\|_{Y^\star (I)} \leq  \|u\|_{Y(I)} \leq & C \left[\|(u_0,u_1)\|_{\dot{H}^1 \times L^2 (\Rm^d)} + \|\phi F(u)\|_{L^1 L^2 (I \times \Rm^d)} \right]\\
  = & C \left[\|(u_0,u_1)\|_{\dot{H}^1 \times L^2 (\Rm^d)} + \|u\|_{Y^\star(I)}^{p_c} \right]
\end{align*}
\end{remark}
\noindent Using the inequalities 
\begin{align*}
 \|\phi F(u)\|_{L^1 L^2 (I \times \Rm^d)} & = \|u\|_{Y^\star (I)}^{p_c};\\
 \|\phi F(u_1) - \phi F(u_2)\|_{L^1 L^2 (I \times \Rm^d)} & \leq C \|u_1 -u_2\|_{Y^\star(I)} \left(\|u_1\|_{Y^\star (I)}^{p_c-1} + \|u_2\|_{Y^\star}^{p_c -1}\right);
\end{align*} 
the Strichartz estimate and a fixed-point argument, we have the following local theory. (Our argument is similar to those in a lot of earlier papers. Therefore we only give important statements but omit most of the proof here. Please see, for instance, \cite{ls} for more details.)

\begin{proposition} [Local solution]
For any initial data $(u_0,u_1) \in \dot{H}^{1} \times L^2 (\Rm^d)$, there is a maximal interval $(-T_{-}(u_0,u_1), T_{+}(u_0,u_1))$ in which the Cauchy problem (CP1) has a solution.
\end{proposition}
\begin{proposition} [Scattering with small data] \label{scattering with small data}
There exists $\delta > 0$ such that if the norm of the initial data $\|(u_0,u_1)\|_{\dot{H}^{1} \times L^2 (\Rm^d)} < \delta$, then the Cauchy problem (CP1) has a global-in-time solution $u$ with $\|u\|_{Y(\Rm)} \lesssim \|(u_0,u_1)\|_{\dot{H}^{1} \times L^2 (\Rm^d)}$.
\end{proposition}
\begin{lemma} [Standard finite blow-up criterion]
Let $u$ be a solution to (CP1) with a maximal lifespan $(-T_-, T_+)$. If $T_{+} < \infty$, 
then $\|u\|_{Y^\star ([0,T_{+}))} = \infty$.
\end{lemma}

\begin{definition} [Scattering] \label{def of scattering}
 We say a solution $u$ to (CP1) with a maximal lifespan $I = (-T_-,T_+)$ scatters in the positive time direction, if $T_+ = \infty$ and there exists a pair $(v_0,v_1) \in \dot{H}^1 \times L^2 (\Rm^d)$, such that 
\[
 \lim_{t \rightarrow + \infty} \left\|\begin{pmatrix} u(\cdot, t)\\ \partial_t u(\cdot, t) \end{pmatrix} - S_L (t) \begin{pmatrix} v_0\\ v_1 \end{pmatrix} \right\|_{\dot{H}^1 \times L^2 (\Rm^d)} = 0. 
\]
In fact, the scattering can be characterized by a more convenient but equivalent condition: $\|u\|_{Y^\star ([T', T_+))} < \infty$, or still equivalently, $\|u\|_{Y([T', T_+))} < \infty$. Here $T'$ is an arbitrary time in $I$.  The scattering at the negative time direction can be defined in the same manner. 
\end{definition}

\begin{theorem} [Long-time perturbation theory, see also \cite{ctao, kenig, kenig1, shen2}] \label{perturbation theory in Y star}
Let $M$ be a positive constant. There exists a constant $\eps_0 = \eps_0 (M)>0$, such that if an approximation solution $\tilde{u}$ defined on $\Rm^d \times I$ ($0\in I$) and a pair of initial data $(u_0,u_1) \in \dot{H}^{1} \times L^2 (\Rm^d)$ satisfy
\begin{align*}
 & (\partial_t^2 - \Delta) (\tilde{u}) - \phi F(\tilde{u}) = e(x,t), \qquad (x,t) \in \Rm^d \times I;\\
 &\|\tilde{u}\|_{Y^\star(I)} < M;\qquad\qquad \|(\tilde{u}(0),\partial_t \tilde{u}(0))\|_{\dot{H}^{1} \times L^2 (\Rm^d)} < \infty;\\
 &\eps \doteq \|e(x,t)\|_{L^1 L^2(I\times \Rm^d)}+ \|S(t)(u_0-\tilde{u}(0),u_1 - \partial_t \tilde{u}(0))\|_{Y^\star(I)} \leq \eps_0;
\end{align*}
then there exists a solution $u(x,t)$ of (CP1) defined in the interval $I$ with the initial data $(u_0,u_1)$ and satisfying
\[
 \|u(x,t) - \tilde{u}(x,t)\|_{Y^\star (I)} \leq C(M) \eps.
\]
\[
 \sup_{t \in I} \left\|\begin{pmatrix} u(t)\\ \partial_t u(t)\end{pmatrix} - \begin{pmatrix} \tilde{u}(t)\\ \partial_t \tilde{u}(t)\end{pmatrix} - S_L (t)\begin{pmatrix} u_0 - \tilde{u}(0)\\ u_1 -\partial_t \tilde{u}(0)\end{pmatrix}
 \right\|_{\dot{H}^{1} \times L^2 (\Rm^d)} \leq C(M)\eps.
\]
\end{theorem}

\begin{proof} 
Let us first prove the perturbation theory when $M$ is sufficiently small. Let $I_1$ be the maximal lifespan of the solution $u(x,t)$ to the
equation (CP1) with the given initial data $(u_0,u_1)$ and assume $[-T_1,T_2] \subseteq I\cap I_1$. By the Strichartz estimates, we have \footnote{The letter C in the argument represents a constant depending solely on the dimension $d$, it may represent different constants in different places.}
\begin{align*}
&\|\tilde{u}-u\|_{Y^\star ([-T_1,T_2])} \\
&\quad \leq \|S_L (t)(u_0-\tilde{u}(0), u_1-\tilde{u}(0))\|_{Y^\star ([-T_1,T_2])} + C
\|e + \phi F (\tilde{u})- \phi F(u)\|_{L^1 L^2([-T_1,T_2] \times \Rm^d)}\\
&\quad \leq \eps + C \|e\|_{L^1 L^2([-T_1,T_2] \times \Rm^d)} + C \|F(\phi^{1/p_c} \tilde{u}) - F(\phi^{1/p_c} u)\|_{L^1 L^2([-T_1,T_2] \times \Rm^d)}\\
&\quad \leq \eps + C \eps + C \|\phi^{1/p_c} (\tilde{u}-u)\|_{Y([-T_1,T_2])} (\|\phi^{1/p_c}\tilde{u}\|_{Y([-T_1,T_2])}^{p_c -1} +
\|\phi^{1/p_c}(\tilde{u}-u)\|_{Y([-T_1,T_2])}^{p_c-1})\\
&\quad \leq C \eps + C \|\tilde{u}-u\|_{Y^\star ([-T_1,T_2])} (M^{p_c-1} + \|\tilde{u}-u\|_{Y^\star ([-T_1,T_2])}^{p_c-1}).
\end{align*}
By a continuity argument in $T_1$ and $T_2$, there exist $M_0 = M_0 (d)>0$ and $\eps_0 = \eps_0 (d) >0$, such that if $M \leq M_0$ and $\eps \leq \eps_0$, we have
\[
 \|\tilde{u}-u\|_{Y^\star ([-T_1,T_2])} \leq C(d) \eps.
\]
Observing that this estimate does not depend on $T_1$ or $T_2$, we are actually able to conclude $I \subseteq I_1$
by the standard blow-up criterion and obtain
\[
 \|\tilde{u}-u\|_{Y^\star (I)} \leq C(d) \eps.
\]
In addition, by the Strichartz estimate we have 
\begin{align*}
 \sup_{t \in I} & \left\|\begin{pmatrix} u(t)\\ \partial_t u(t)\end{pmatrix} - \begin{pmatrix} \tilde{u}(t)\\ \partial_t \tilde{u}(t)\end{pmatrix}
 - S_L (t) \begin{pmatrix} u_0 - \tilde{u}(0)\\ u_1 -\partial_t \tilde{u}(0)\end{pmatrix} \right\|_{\dot{H}^{1} \times L^2 (\Rm^d)}\\
 &\leq C \| \phi F(u) - \phi F(\tilde{u}) - e\|_{L^1 L^2 (I\times \Rm^d)}\\
 &\leq C \left(\|e\|_{L^1 L^2 (I\times \Rm^d)} + \|F(\phi^{1/p_c} u) - F(\phi^{1/p_c} \tilde{u})\|_{L^1 L^2 (I\times \Rm^d)}\right)\\
 &\leq C \left[\eps + \|u -\tilde{u}\|_{Y^\star (I)}\left(\|\tilde{u}\|_{Y^\star (I)}^{p_c-1} + \|u-\tilde{u}\|_{Y^\star (I)}^{p_c-1}\right)\right]\\
 &\leq C \eps.
\end{align*}
This finishes the proof as $M \leq M_0$. To deal with the general case, we can separate the time interval $I$ into finite number of subintervals $\{I_j\}_{1\leq j\leq n}$, so that $\|\tilde{u}\|_{Y^{\star}(I_j)} < M_0$, and then iterate our argument above.
\end{proof}

\begin{remark} \label{lifespan lower bound for compact set}
If $K$ is a compact subset of the space $\dot{H}^{1} \times L^2 (\Rm^d)$, then there exists $T = T(K) > 0$ such that for any $(u_0,u_1) \in K$, we have $T_{+}(u_0,u_1) > T(K)$ and $T_{-}(u_0,u_1) > T(K)$. This is a direct corollary from the perturbation theory.
\end{remark}

\subsection{Ground States and the Energy Trapping}

In this subsection we make a brief review on the properties of ground states for the equation (CP0) and understand the ``energy trapping'' phenomenon. Let us first recall a particular ground state
\[
  W(x) = \frac{1}{\left(1 + \frac{|x|^2}{d(d-2)}\right)^{\frac{d-2}{2}}}.
\]
The ground state is not unique, because given any constants $\lambda \in \Rm^+$ and $x_0 \in \Rm^d$, the function 
\[
 W_{\lambda, x_0} (x) = \frac{1}{\lambda^{\frac{d-2}{2}}} W \left(\frac{x-x_0}{\lambda}\right)
\]
is also a ground state. Any ground state constructed in this way share the same $\dot{H}^1$ and $L^{2^\ast}$ norms as $W$. The ground state $W$, or any other ground state we constructed above, can be characterized by the following lemma. 

\begin{lemma}[Please see \cite{best constant}] \label{best const}
 The function $W$ gives the best constant $C_d$ in the Sobolev embedding $\dot{H}^1(\Rm^d) \hookrightarrow L^{2^\ast}(\Rm^d)$. Namely, the inequality 
 \[ 
  \|u\|_{L^{2^\ast}} \leq C_d \|\nabla u\|_{L^2}
 \]
holds for any function $u \in \dot{H}^1 (\Rm^d)$ and becomes an equality for $u = W$. 
\end{lemma}
\begin{remark}
Because the function $W$ is a solution to $-\Delta W = |W|^{\frac{4}{d-2}} W$, we also have 
\[
 \int_{\Rm^d} |\nabla W|^2 dx = \int_{\Rm^d} |W|^{2^\ast} dx \Longrightarrow \|\nabla W\|_{L^2}^2 = \|W\|_{L^{2^\ast}}^{2^\ast} = C_d^{2^\ast} \|\nabla W\|_{L^2}^{2^\ast}
\]
As a result, we have $C_d^{2^\ast} \|\nabla W\|_{L^2}^{2^\ast - 2} = 1$ and $E_1 (W,0)  = (1/d) \|\nabla W\|_{L^2}^2 > 0$. 
\end{remark}

\begin{proposition} [Energy Trapping, see also \cite{kenig, kenig1}] \label{energy trapping}
Let $\delta >0$. If $u$ is a solution to (CP1) in the focusing case with initial data $(u_0,u_1)$ so that 
\begin{align*}
 &\|\nabla u_0\|_{L^2} < \|\nabla W\|_{L^2},& &E_\phi (u_0,u_1) < (1 -\delta) E_1 (W,0);&
\end{align*}
then for any time $t$ in the maximal lifespan $I$ of $u$ we have 
\begin{align}
 &\|(u(\cdot,t), \partial_t u(\cdot, t))\|_{\dot{H}^1 \times L^2 (\Rm^d)} < (1 - 2\delta/d)^{1/2} \|\nabla W\|_{L^2} \label{energy trap 1} \\
  &\int_{\Rm^d} \left(|\nabla u(x,t)|^2 - |u(x,t)|^{2^\ast}\right) dx \simeq_\delta \int_{\Rm^d}   |\nabla u(x,t)|^2  dx \label{energy trap 2} \\
 &\int_{\Rm^d} |u_t (x,t)|^2 dx + \int_{\Rm^d} \left(|\nabla u(x,t)|^2 - |u(x,t)|^{2^\ast}\right) dx \simeq_\delta E_\phi (u_0,u_1). \label{energy trap 3} 
\end{align}
\end{proposition}
\begin{proof} If we have $\|\nabla u(\cdot, t_0)\|_{L^2} < \|\nabla W\|_L^2$ for some time $t_0 \in I$, then we have 
 \[
  \|u(\cdot, t_0)\|_{L^{2^\ast}} \leq C_d \|\nabla u(\cdot, t_0)\|_{L^2} < C_d \|\nabla W\|_{L^2} = \|W\|_{L^{2^\ast}}. 
 \]
 Therefore 
 \begin{align*}
  \frac{1}{2}\|(u(\cdot,t_0), \partial_t u(\cdot, t_0))\|_{\dot{H}^1 \times L^2 (\Rm^d)}^2 & = E_{\phi} (u_0,u_1) + \frac{1}{2^\ast} \int_{\Rm^d} \phi |u(x,t_0)|^{2^\ast} dx \nonumber \\
   & < (1-\delta) E_1(W,0) + \frac{1}{2^\ast} \int_{\Rm^d}  |W|^{2^\ast} dx\nonumber\\
   & = \frac{1}{2}\|\nabla W\|^2 - \delta E_1 (W,0)\nonumber \\
   & = \left(\frac{1}{2} - \frac{\delta}{d}\right) \|\nabla W\|^2.
 \end{align*}
Because we know $(u ,\partial_t u) \in C(I; \dot{H}^1 \times L^2 (\Rm^d))$, this implies that the inequality \eqref{energy trap 1} holds for each $t$ in a small neighbourhood of $t_0$. Using a continuity argument and the fact $\|\nabla u_0\|_{L^2} < \|\nabla W\|_L^2$, we know the inequality \eqref{energy trap 1} holds for all $t\in I$. Applying this inequality and the Sobolev embedding we have 
\begin{align}
 \int_{\Rm^d} |u(x,t)|^{2^\ast} dx & =  \|u (\cdot,t)\|_{L^{2^\ast}}^{2^\ast} \leq C_d^{2^\ast} \|\nabla u (\cdot, t)\|_{L^2}^{2^\ast} = C_d^{2^\ast} \|\nabla u (\cdot, t)\|_{L^2}^{2^\ast-2} \int_{\Rm^d}   |\nabla u(x,t)|^2  dx\nonumber \\
 & \leq C_d^{2^\ast} (1 - 2\delta/d)^{(2^\ast-2)/2}  \|\nabla W\|_{L^2}^{2^\ast-2} \int_{\Rm^d}  |\nabla u(x,t)|^2  dx \label{L6 H1}\\
 & \leq (1 - 2\delta/d)^{(2^\ast-2)/2} \int_{\Rm^d}  |\nabla u(x,t)|^2  dx\nonumber
\end{align}
Here we use the identity $C_d^{2^\ast} \|\nabla W\|_{L^2}^{2^\ast - 2} = 1$. Since the constant $(1 - 2\delta/d)^{(2^\ast-2)/2} < 1$, we obtain the estimate \eqref{energy trap 2}. The estimate \eqref{energy trap 3} immediately follows as a direct corollary.  
\end{proof}
\begin{remark} \label{positive energy} 
 If $(u_0,u_1) \in \dot{H}^1 \times L^2 (\Rm^d)$ satisfies $\|\nabla u_0\|_{L^2} < \|\nabla W\|_{L^2}$, then as we did in \eqref{L6 H1}, we can show 
\[
 \int_{\Rm^d} |u(x,t)|^{2^\ast} dx \leq  \int_{\Rm^d}   |\nabla u(x,t)|^2  dx.
\]
Therefore we have $ E_\phi (u_0,u_1) \simeq E_1 (u_0,u_1) \simeq E_c (u_0,u_1) \simeq \|(u_0,u_1)\|_{\dot{H}^1 \times L^2 (\Rm^d)}^2$.
Here $c$ is any constant in $(0,1)$.  Furthermore, if any of the energy above is zero, then we have $(u_0,u_1) = 0$. 
\end{remark}

\subsection{Known Results with a Pure Power-type Nonlinearity}

\paragraph{The defocusing case} The problem about the global behaviour of solutions in the defocusing, energy-critical case with a pure power-type nonlinearity was completely solved by Grillakis \cite{mg1, mg2} and Shatah-Struwe \cite{ss1,ss2} in 1990's. 
\begin{theorem}
 Let $(u_0,u_1) \in \dot{H}^1 \times L^2 (\Rm^d)$ with $3 \leq d \leq 5$. Then the solution to the Cauchy problem 
\[
 \left\{\begin{array}{ll}
 \partial_t^2 u - \Delta u = - |u|^{4/(d-2)} u, & (x,t) \in \Rm^d \times \Rm;\\
 u (\cdot, 0) = u_0; &\\
 \partial_t u(\cdot, 0) = u_1; &
 \end{array}\right.
\]
exists globally in time and scatters. 
\end{theorem}

\paragraph{The focusing case} The global behaviour of solutions in the focusing case is much more complicated and subtle. It has not been completely understood. The following result is the main theorem in \cite{kenig}, on the global well-posedness, scattering and blow-up of solutions to the focusing, energy-critical non-linear wave equation, as we mentioned in the introduction. 

\begin{theorem} \label{critical focusing phi equal 1}
 Let $(u_0,u_1) \in \dot{H}^1 \times L^2 (\Rm^d)$ with $3 \leq d \leq 5$. Assume that $E_1 (u_0,u_1) < E_1 (W,0)$. Let $u$ be the corresponding solution to the Cauchy problem (CP0) with a maximal interval of existence $I = (- T_- (u_0,u_1), T_+ (u_0,u_1))$. Then 
\begin{itemize}
 \item[(i)] If $\|\nabla u_0\|_{L^2} < \|\nabla W\|_{L^2}$, then $I = \Rm$ and $u$ scatters in both time directions. 
 \item[(ii)] If $\|\nabla u_0\|_{L^2} > \|\nabla W\|_{L^2}$, then $u$ blows up within finite time in both two directions, namely 
 \begin{align*}
  &T_- (u_0,u_1) < +\infty;& &T_+ (u_0,u_1) < +\infty.&
 \end{align*}
\end{itemize}
\end{theorem}

\paragraph{Nonlinearity with a coefficient} Assume that $c$ is a positive constant. If $u$ is a solution to the equation
\begin{equation} \label{cp0 with c}
\left\{\begin{array}{l}
 \partial_t^2 u - \Delta u = c |u|^{4/(d-2)} u,\\
 (u, \partial_t u)|_{t =0} = (u_0,u_1); \end{array}\right.
\end{equation}
then $c^{\frac{d-2}{4}} u$ is a solution to the equation (CP0) with initial data $(c^{\frac{d-2}{4}} u_0, c^{\frac{d-2}{4}} u_1)$. This transformation immediately gives us 
 
\begin{corollary} \label{critical focusing phi equal c}
 Let $(u_0,u_1) \in \dot{H}^1 \times L^2 (\Rm^d)$ with $3 \leq d \leq 5$. Assume that $E_c (u_0,u_1) < c^{-\frac{d-2}{2}} E_1 (W,0)$. Let $u$ be the corresponding solution to the Cauchy problem \eqref{cp0 with c} with a maximal interval of existence $I = (- T_- (u_0,u_1), T_+ (u_0,u_1))$. Then 
\begin{itemize}
 \item[(i)] If $\|\nabla u_0\|_{L^2} < c^{-\frac{d-2}{4}} \|\nabla W\|_{L^2}$, then $I = \Rm$ and $u$ scatters in both time direction. 
 \item[(ii)] If $\|\nabla u_0\|_{L^2} > c^{-\frac{d-2}{4}}  \|\nabla W\|_{L^2}$, then $u$ blows up within finite time in both two directions, namely 
 \begin{align*}
  &T_- (u_0,u_1) < +\infty;& &T_+ (u_0,u_1) < +\infty.&
 \end{align*}
\end{itemize}
\end{corollary}

\subsection{Technical Lemma}

\begin{lemma}\label{weak convergence of w}
Let $\{(w_{0,n}, w_{1,n})\}_{n \in {\mathbb Z}^+}$ be a sequence in the space $\dot{H}^1 \times L^2 (\Rm^d)$. If the norms $\|(w_{0,n}, w_{1,n})\|_{\dot{H}^1 \times L^2 (\Rm^d)}$ are uniformly bounded and $\|S_L (t)(w_{0,n}, w_{1,n})\|_{Y^\star (\Rm)} \rightarrow 0$ as $n \rightarrow \infty$, then the pairs $(w_{0,n}, w_{1,n})$ weakly converge to $0$ in $\dot{H}^1 \times L^2 (\Rm^d)$.
\end{lemma}
\begin{proof}
 If the weak limit $(w_{0,n}, w_{1,n}) \rightharpoonup 0$ were not true, we could assume $(w_{0,n}, w_{1,n}) \rightharpoonup (w_0,w_1) \neq 0$ in $\dot{H}^1 \times L^2 (\Rm^d)$ by possibly passing to a subsequence. As a result, we have $S_L (t)(w_{0,n}, w_{1,n}) \rightharpoonup S_L (t)(w_0,w_1)$ in the space $Y^\star (\Rm)$. Thus we have $S_L (t)(w_0,w_1) = 0 \Longrightarrow (w_0,w_1) = 0$. This is a contradiction.  
\end{proof}

\section{Profile Decomposition}

In this section we discuss the profile decomposition and do the preparation work for our compactness procedure. In order to save space we use the notation $H$ for the space $\dot{H}^1 \times L^2 (\Rm^d)$ when it is necessary.

\subsection{Linear Profile Decomposition}

\begin{theorem} [Linear Profile Decomposition] \label{profile decomposition}
Given a sequence $(u_{0,n}, u_{1,n}) \in H$ so that $\|(u_{0,n}, u_{1,n})\|_H \leq A$, these exist a subsequence of $(u_{0,n}, u_{1,n})$ (We still use the notation $(u_{0,n}, u_{1,n})$ for the subsequence), a sequence of free waves $V_j(x,t) = S_L (t) (v_{0,j}, v_{1,j})$, a family of triples $(\lambda_{j,n}, x_{j,n}, t_{j,n}) \in \Rm^+ \times \Rm^3 \times \Rm$, which are ``almost orthogonal", i.e. we have 
\[
 \lim_{n \rightarrow +\infty} \left( \frac{\lambda_{j',n}}{\lambda_{j,n}}+ \frac{\lambda_{j,n}}{\lambda_{j',n}} + \frac{|x_{j,n} - x_{j',n}|}{\lambda_{j,n}}+\frac{|t_{j,n} - t_{j',n}|}{\lambda_{j,n}}\right) = + \infty
\]
for $j \neq j'$; such that
\begin{itemize}
\item [(I)] We have the following decomposition for each fixed $J\geq 1$,
\[
 (u_{0,n}, u_{1,n})  = \sum_{j =1}^J (V_{j,n}(\cdot, 0), \partial_t V_{j,n}(\cdot, 0)) + (w_{0,n}^J, w_{1,n}^J).
\]
Here $V_{j,n}$ is another free wave derived from $V_j$:
\begin{align*}
(V_{j,n}(\cdot, t), \partial_t V_{j,n}(\cdot, t)) &= \left(\frac{1}{\lambda_{j,n}^{\frac{d-2}{2}}} V_j\left(\frac{x - x_{j,n}}{\lambda_{j,n}}, \frac{t-t_{j,n}}{\lambda_{j,n}}\right),
 \frac{1}{\lambda_{j,n}^{d/2}} \partial_t V_j\left(\frac{x-x_{j,n}}{\lambda_{j,n}}, \frac{t-t_{j,n}}{\lambda_{j,n}}\right)\right)\\
& = S_L (t -t_{j,n}) \Di_{\lambda_{j,n},  x_{j,n}} (v_{0,j}, v_{1,j});
\end{align*}
\item [(II)] We have the following limits regarding the remainder $(w_{0,n}^J, w_{1,n}^J)$ as $J \rightarrow \infty$.  
\begin{align*}
  &\limsup_{n \rightarrow \infty} \|S_L (t)(w_{0,n}^J, w_{1,n}^J)\|_{Y(\Rm)} \rightarrow 0; & &\limsup_{n \rightarrow \infty} \|S_L (t) (w_{0,n}^J, w_{1,n}^J)\|_{L^{\infty} L^{2^\ast} (\Rm \times \Rm^d)} \rightarrow 0; &
\end{align*}
\item[(III)] For each given $J \geq 1$, we have (the error $o_J(n) \rightarrow 0$ as $n \rightarrow \infty$)
\[
 \|(u_{0,n}, u_{1,n})\|_H^2 = \sum_{j=1}^J \|V_j\|_H^2 + \|(w_{0,n}^J, w_{1,n}^J)\|_H^2 + o_J(n).
\]
Here the notation $\|V_j\|_H$ represents $\|(V_j(\cdot, t), \partial_t V_j (\cdot,t))\|_H$ which is a constant independent of $t \in \Rm$. 
\end{itemize}
\end{theorem}
Please see \cite{bahouri} for a proof. There are a few remarks. First of all, the original paper is only for the three-dimensional case but the same argument also works for higher dimensions. Second, only the limit with $Y$ norm in part (II) of the conclusion is mentioned in the original theorem. However, we can substitute $Y$ norm by any Strichartz norm $L^q L^r (\Rm \times \Rm^d)$ with $q >2$, as mentioned in Page 136 of the paper. Here we use the $L^\infty L^{2^\ast}$ norm. Finally, the original paper proves the theorem under an additional assumption labelled (1.6) there. But this condition can be eliminated according to Remark 5 on Page 159 of the paper.  

\begin{remark}
Passing to a subsequence if necessary, we can always assume the following limits hold as $n \rightarrow \infty$ in addition.
\begin{align*}
 &\lambda_{j,n} \rightarrow \lambda_j \in [0, \infty], & &\frac{-t_{j,n}}{\lambda_{j,n}} \rightarrow t_j \in [-\infty, \infty],& &x_{j,n} \rightarrow x_j \in \Rm^3 \cup \{\infty\}&
\end{align*} 
Here $x_{j,n} \rightarrow \infty$ means $|x_{j,n}| \rightarrow \infty$. Furthermore, by adjusting the free waves $V_j$'s, we can always assume each of the triples $(\lambda_j, x_j, t_j)$ satisfies one of the following conditions
\begin{itemize}
 \item[(I)] $\lambda_j = 0$;
 \item[(II)] $\lambda_j = 1$, $x_j$ is either $0$ or $\infty$;
 \item[(III)] $\lambda_j = + \infty$. 
\end{itemize}
\end{remark}

\begin{remark} \label{almost orthogonality in H}
A basic computation shows ($j\neq j'$)
\begin{align*}
& \left\langle \begin{pmatrix} V_{j,n} (\cdot, t_0)\\ \partial_t V_{j,n} (\cdot,t_0) \end{pmatrix}, \begin{pmatrix} V_{j',n} (\cdot, t_0)\\ \partial_t V_{j',n} (\cdot,t_0) \end{pmatrix} \right \rangle_H\\  = &\left\langle S_L (t_0 -t_{j,n}) \Di_{\lambda_{j,n},  x_{j,n}} \begin{pmatrix} v_{0,j}\\ v_{1,j} \end{pmatrix}, S_L (t_0 -t_{j',n}) \Di_{\lambda_{j',n},  x_{j',n}} \begin{pmatrix} v_{0,j'}\\ v_{1,j'}\end{pmatrix} \right\rangle_H\\
 = &\left\langle \Di_{\lambda_{j,n}/\lambda_{j',n}} \begin{pmatrix} v_{0,j}\\ v_{1,j} \end{pmatrix}, \Di_{1, \frac{x_{j',n}-x_{j,n}}{\lambda_{j',n}}} S_L \left(\frac{t_{j,n} -t_{j',n}}{\lambda_{j',n}}\right)\begin{pmatrix} v_{0,j'}\\ v_{1,j'}\end{pmatrix}\right\rangle_H.
\end{align*}
Since the triples $(\lambda_{j,n}, x_{j,n}, t_{j,n})$ and $(\lambda_{j',n}, x_{j',n}, t_{j',n})$ are almost orthogonal, we have 
\[
 \lim_{n\rightarrow \infty} \left\langle \begin{pmatrix} V_{j,n} (\cdot, t_0)\\ \partial_t V_{j,n} (\cdot,t_0) \end{pmatrix}, \begin{pmatrix} V_{j',n} (\cdot, t_0)\\ \partial_t V_{j',n} (\cdot,t_0) \end{pmatrix} \right \rangle_H = 0
\]
for each $t_0 \in \Rm$ and $j \neq j'$.
\end{remark}

\begin{lemma} \label{L6 tends to zero}
If $V(x,t) = S_L (t)(v_0,v_1)$ be a solution to the linear wave equation with a finite energy, then we have the limit 
\[
 \lim_{t \rightarrow \pm \infty} \|V(\cdot,t)\|_{L^{2^\ast}(\Rm^d)}  =0.
\]
\end{lemma}
\begin{proof} 
 By the Sobolev embedding $\dot{H}^1 (\Rm^d) \hookrightarrow L^{2^\ast} (\Rm^d)$ and the fact that the linear propagation preserves the $\dot{H}^1 \times L^2$ norm, it is sufficiently to prove this lemma for each $(v_0,v_1)$ in a dense subset of $\dot{H}^1 \times L^2$. Thus we only need to consider smooth and compactly supported initial data. In this case the kernel of the wave propagation gives the well-known estimate 
 \[
  \|V(\cdot, t)\|_{L^\infty (\Rm^d)} \lesssim |t|^{-\frac{d-1}{2}}.
 \]
 On the other hand the volume of the support of $V(\cdot,t)$ satisfies $\left|\hbox{Supp} (V(\cdot ,t))\right| \lesssim |t|^d$ when $|t|$ is large. This immediately gives the estimate 
$\|V(\cdot,t)\|_{L^{2^\ast}} \lesssim |t|^{-1/2}$ and finishes the proof. 
\end{proof}

\begin{lemma} \label{limit of phi V L6}
For each $j$ we have the limit
\[
  \lim_{n \rightarrow \infty} \int_{\Rm^d} \phi |V_{j,n}(x,0)|^{2^\ast} dx = \left\{\begin{array}{ll} 
  0, & \hbox{if}\; t_j = \pm \infty\; \hbox{or}\; \lambda_{j} = +\infty\; \hbox{or}\; x_{j} = \infty;\\
  \int_{\Rm^d} \phi |V_j (x,t_j)|^{2^\ast} dx, & \hbox{if}\; t_j \in \Rm,\;  \lambda_j = 1\; \hbox{and}\; x_j = 0;\\
  \int_{\Rm^d} \phi(x_j) |V_j(x,t_j)|^{2^\ast} dx, & \hbox{if}\; t_j \in \Rm,\; \lambda_j = 0\; \hbox{and}\; x_j \in \Rm^d.
  \end{array}\right.
\]
\end{lemma}
\begin{proof}
By Lemma \ref{L6 tends to zero}, the case $t_j = \pm \infty$ is trivial. Thus we only need to consider the case $t_j$ is finite. By our assumption $-t_{j,n}/\lambda_{j,n} \rightarrow t_j$, the fact $V_j (\cdot, t) \in C(\Rm; \dot{H}^1 (\Rm^d))$ and the Sobolev embedding $\dot{H}^1 \hookrightarrow L^{2^\ast}$, we have
\[
 \left\|V_{j,n}(\cdot,0) - \frac{1}{\lambda_{j,n}^{\frac{d-2}{2}}} V_j \left(\frac{\cdot - x_{j,n}}{\lambda_{j,n}}, t_j\right)\right\|_{L^{2^\ast} (\Rm^d)} = \left\|V_j\left(\cdot, \frac{-t_{j,n}}{\lambda_{j,n}}\right) - V_j\left(\cdot, t_j\right)\right\|_{L^{2^\ast}(\Rm^d)} \rightarrow 0.
\]
This implies $\displaystyle \lim_{n\rightarrow \infty}  \left\|V_{j,n}(\cdot,0)\right\|_{L^{2^\ast}(\Rm^d; \phi dx)} = \lim_{n\rightarrow \infty}  \left\|\frac{1}{\lambda_{j,n}^{\frac{d-2}{2}}} V_j \left(\frac{\cdot - x_{j,n}}{\lambda_{j,n}}, t_j\right)\right\|_{L^{2^\ast}(\Rm^d; \phi dx)}$. 
Thus we have 
\begin{align*}
 \lim_{n \rightarrow \infty} \int_{\Rm^d} \phi |V_{j,n}(x,0)|^{2^\ast} dx = & \lim_{n \rightarrow \infty} \int_{\Rm^d} \phi \left|\frac{1}{\lambda_{j,n}^{\frac{d-2}{2}}} V_j \left(\frac{x - x_{j,n}}{\lambda_{j,n}}, t_j\right)\right|^{2^\ast} dx\\
  = & \lim_{n \rightarrow \infty} \int_{\Rm^d} \phi(\lambda_{j,n} x + x_{j,n} ) |V_j (x,t_j)|^{2^\ast} dx.
\end{align*}
The limit can be evaluated as 
\begin{itemize} 
 \item If $\lambda_j = \infty$, then we can conclude that the limit is zero, by observing $|V_j (x,t_j)|^{2^\ast} \in L^1 (\Rm^d)$, $0 \leq \phi(x) \leq 1$ and the fact that for any $\eps_0 > 0$, we have 
 \[
  \lim_{n \rightarrow \infty} \left|\left\{x: |\phi(\lambda_{j,n} x + x_{j,n} )| > \eps_0\right\}\right| = 0.
 \]
 \item If $\lambda_j < \infty$, then the limit can be evaluated by the dominated convergence theorem, because for any fixed $x \in \Rm^d$ we have 
 \[
  \lim_{n \rightarrow \infty} \phi(\lambda_{j,n} x + x_{j,n}) = \left\{\begin{array}{ll} 
  0, & \hbox{if}\; \lambda_{j} < +\infty\; \hbox{and}\; x_{j} = \infty;\\
  \phi(x) , & \hbox{if}\; \lambda_j = 1\; \hbox{and}\; x_j = 0;\\
  \phi(x_j), & \hbox{if}\; \lambda_j = 0\; \hbox{and}\; x_j \in \Rm^d.
  \end{array}
  \right. 
 \]
\end{itemize}
\end{proof}

\begin{lemma} \label{L6 ortho 1} 
Assume $\tilde{V}_1, \tilde{V}_2 \in L^{2^\ast} (\Rm^d)$. If the pairs $\{(\lambda_{j,n}, x_{j,n})\}_{j=1,2; n \in {\mathbb Z}^+}$ satisfy
 \[
  \lim_{n \rightarrow +\infty} \left(\frac{\lambda_{2,n}}{\lambda_{1,n}}+ \frac{\lambda_{1,n}}{\lambda_{2,n}} + \frac{|x_{1,n} - x_{2,n}|}{\lambda_{1,n}}\right) = \infty,
 \]
then we have the following limit as $n \rightarrow \infty$
\[
 N(n) = \left\|\frac{1}{\lambda_{1,n}^{(d-2)/2}} \tilde{V}_1 \left(\frac{x - x_{1,n}}{\lambda_{1,n}}\right)\cdot \frac{1}{\lambda_{2,n}^{(d-2)/2}} \tilde{V}_{2} \left(\frac{x - x_{2,n}}{\lambda_{2,n}}\right)\right\|_{L^{2^\ast/2}} \rightarrow 0.
\]
\end{lemma}
\begin{proof} 
Let us first define
\begin{align*}
 \tilde{V}_{j,n} (x) = \frac{1}{\lambda_{j,n}^{(d-2)/2}} \tilde{V}_{j} \left(\frac{x - x_{j,n}}{\lambda_{j,n}}\right).
\end{align*}
Observing the continuity of the map
\[
 \Phi: L^{2^\ast} (\Rm^d) \times L^{2^\ast} (\Rm^d) \rightarrow l^\infty, \qquad \Phi (\tilde{V}_1,\tilde{V}_2) = \left\{\left\| \tilde{V}_{1,n}  \tilde{V}_{2,n} \right\|_{L^{2^\ast/2}}\right\}_{n \in {\mathbb Z}^+},
\]
we can also assume, without loss of generality, that 
\begin{align*}
 &|\tilde{V}_j (x)| \leq M_j, \; \hbox{for any}\; x \in \Rm^d; & &\hbox{Supp} (\tilde{V}_j) \subseteq \{x: |x| < R_j\} &
\end{align*}
for $j = 1,2$ and some constants $M_j$ and $R_j$, because the functions satisfying these conditions are dense in the space $L^{2^\ast} (\Rm^d)$.  If the conclusion were false, we would find a sequence $n_1 < n_2 < n_3 < \cdots$ and a positive constant $\eps_0$ such that $ N(n_k) \geq \eps_0$. There are three cases
\begin{itemize}
 \item [(I)] $\limsup_{k \rightarrow \infty} \lambda_{1, n_k}/ \lambda_{2,n_k} = \infty$. Let us consider the product $\tilde{V}_{1,n_k} \tilde{V}_{2,n_k}$. This function is supported in a $d$-dimensional ball centred at $x_{2,n_k}$ with a radius $\lambda_{2,n_k} R_2$ since $\tilde{V}_{2,n_k}$ is supported in this ball. On the other hand, we have 
 \[
  \left|\tilde{V}_{1,n_k} \tilde{V}_{2,n_k}\right| \leq \lambda_{1,n_k}^{-\frac{d-2}{2}} \lambda_{2,n_k}^{-\frac{d-2}{2}} M_1 M_2.
 \] 
A basic computation shows 
\[
N(n_k) = \left\|\tilde{V}_{1,n_k} \tilde{V}_{2,n_k} \right\|_{L^{2^\ast/2} (\Rm^d)} \leq C(d) M_1 M_2 R_2^{d-2} \left(\frac{\lambda_{2, n_k}}{\lambda_{1,n_k}}\right)^{\frac{d-2}{2}}. 
\]
The upper bound tends to zero as $\lambda_{1, n_k}/ \lambda_{2,n_k} \rightarrow \infty$. Thus we have a contradiction. 
\item [(II)] $\limsup_{k \rightarrow \infty} \lambda_{2, n_k}/ \lambda_{1,n_k} = \infty$. This can be handled in the same way as case (I).
\item [(III)] $\lambda_{1, n_k} \simeq \lambda_{2, n_k}$. Thus we also have $\displaystyle \frac{|x_{1,n_k} - x_{2,n_k}|}{\lambda_{1,n_k}}\rightarrow \infty$. This implies $\hbox{Supp}(\tilde{V}_{1,n_k}) \cap \hbox{Supp}(\tilde{V}_{2,n_k}) = \emptyset$ when $k$ is sufficiently large thus gives a contradiction. 
\end{itemize}
\end{proof}
\begin{lemma} \label{L6 almost orthogonality}
 For any $j \neq j'$, we have the limit 
 \[
  \lim_{n \rightarrow \infty} \|V_{j,n}(\cdot, 0) V_{j',n}(\cdot, 0)\|_{L^{2^\ast/2}} = 0.
 \]
\end{lemma}
\begin{proof}
 If $t_j = \pm \infty$, then Lemma \ref{L6 tends to zero} guarantees 
 \[
  \|V_{j,n}(\cdot, 0)\|_{L^{2^\ast}(\Rm^d)} = \left\|V_j \left(\cdot, \frac{-t_{j,n}}{\lambda_{j,n}}\right)\right\|_{L^{2^\ast}(\Rm^d)} \rightarrow 0
 \]
 On the other hand, we know 
 \begin{equation} \label{uniform bound of L6 1}
  \|V_{j',n}(\cdot, 0)\|_{L^{2^\ast}(\Rm^d)} \lesssim \|\nabla V_{j',n}(\cdot,0)\|_{L^2} \leq \|V_j\|_H \leq A
 \end{equation}
 This immediately finishes our proof. Thus we only need to consider the case that $t_j, t_{j'}$ are both finite. In this case the almost orthogonality of the triples gives 
 \[
   \lim_{n \rightarrow +\infty} \left(\frac{\lambda_{j,n}}{\lambda_{j',n}}+ \frac{\lambda_{j',n}}{\lambda_{j,n}} + \frac{|x_{j,n} - x_{j',n}|}{\lambda_{j,n}}\right) = \infty.
 \]
By the fact $V_j(\cdot,t), V_{j'}(\cdot,t) \in C(\Rm; L^{2^\ast} (\Rm^d))$, we have 
\begin{align*}
 \left\|V_{j,n} (\cdot, 0) -  \frac{1}{\lambda_{j,n}^{(d-2)/2}} V_j \left(\frac{x - x_{j,n}}{\lambda_{j,n}}, t_j\right) \right\|_{L^{2^\ast}} & = \left\|V_j \left(\cdot, \frac{-t_{j,n}}{\lambda_{j,n}}\right) - V_j (\cdot, t_j)\right\|_{L^{2^\ast}} \rightarrow 0;\\
 \left\|V_{j',n} (\cdot, 0) -  \frac{1}{\lambda_{j',n}^{(d-2)/2}} V_{j'} \left(\frac{x - x_{j',n}}{\lambda_{j',n}}, t_j'\right) \right\|_{L^{2^\ast}} & = \left\|V_{j'} \left(\cdot, \frac{-t_{j',n}}{\lambda_{j',n}}\right) - V_{j'} (\cdot, t_{j'})\right\|_{L^{2^\ast}} \rightarrow 0;
\end{align*}
Therefore we only need to show
\[
 \lim_{n \rightarrow \infty}   \left\|\frac{1}{\lambda_{j,n}^{(d-2)/2}} V_j \left(\frac{x - x_{j,n}}{\lambda_{j,n}}, t_j\right)\cdot \frac{1}{\lambda_{j',n}^{(d-2)/2}} V_{j'} \left(\frac{x - x_{j',n}}{\lambda_{j',n}}, t_{j'}\right)\right\|_{L^{2^\ast/2}} = 0.
\]
This immediately follows Lemma \ref{L6 ortho 1}. 
\end{proof}
\begin{lemma} 
The profile decomposition we obtain above satisfies 
\begin{equation} \label{sum of L6 limit}
  \lim_{n \rightarrow \infty} \int_{\Rm^d} \phi |u_{0,n} (x)|^{2^\ast} dx =   \sum_{j=1}^\infty \lim_{n \rightarrow \infty} \int_{\Rm^d} \phi |V_{j,n} (x,0)|^{2^\ast} dx.
\end{equation}
\end{lemma}
\begin{proof} First of all, we know each limit on the right hand of \eqref{sum of L6 limit} exists by Lemma \ref{limit of phi V L6}. Let us first show 
\begin{equation} \label{partial sum for L6}
 \lim_{n \rightarrow \infty} \int_{\Rm^d} \phi \left|\sum_{j=1}^J V_{j,n} (x,0) \right|^{2^\ast} dx = \sum_{j=1}^J  \lim_{n \rightarrow \infty} \int_{\Rm^d} \phi \left| V_{j,n} (x,0) \right|^{2^\ast} dx.
\end{equation}
This can proved via an induction by observing (take the value of $V_{j,n}$ at $t=0$ and let $G(u) = |u|^{2^\ast}$)
\begin{align*}
 & \limsup_{n \rightarrow \infty} \int_{\Rm^d} \left| G\left(\sum_{j=1}^J V_{j,n} (\cdot,0)\right) -  G\left(\sum_{j=1}^{J-1} V_{j,n}(\cdot,0)\right) - G(V_{J,n}(\cdot,0))\right| \phi dx\\
 \leq & \limsup_{n \rightarrow \infty}  \int_{\Rm^d} \left|\left[V_{J,n} \int_0^1 G'\left(\tau V_{J,n} + \sum_{j=1}^{J-1} V_{j,n}\right) d\tau \right]
- \left[V_{J,n} \int_0^1 G'(\tau V_{J,n})d\tau\right] \right|dx\\
= & \limsup_{n \rightarrow \infty} \int_{\Rm^d} \left| \left[V_{J,n} \sum_{j=1}^{J-1} V_{j,n}\right] \left[ \int_0^1 \int_0^1
G''\left(\tau V_{J,n} + \tilde{\tau}\sum_{j=1}^{J-1} V_{j,n}\right) d\tilde{\tau} d\tau\right]\right| dx\\
\leq& \limsup_{n \rightarrow \infty} C \left(\sum_{j=1}^{J-1} \|V_{J,n} V_{j,n}\|_{L^{2^\ast/2} (\Rm^d)} \right)
\left( \sum_{j=1}^{J}\|V_{j,n}\|_{L^{2^\ast} (\Rm^d)} \right)^{2^\ast -2}\\
= & 0.
\end{align*}
Here we use Lemma \ref{L6 almost orthogonality} and the estimate \eqref{uniform bound of L6 1}. Next we can rewrite \eqref{partial sum for L6} into 
\[
 \left(\sum_{j=1}^J  \lim_{n \rightarrow \infty} \int_{\Rm^d} \phi \left| V_{j,n} (x,0) \right|^{2^\ast} dx\right)^{1/2^{\ast}} = \lim_{n \rightarrow \infty}
 \left\|\sum_{j=1}^J V_{j,n} (x,0)\right\|_{L^{2^\ast} (\Rm^d; \phi dx)}
\]
The conclusion (II) of profile decomposition gives 
\[
 \limsup_{n \rightarrow \infty} \left\| w_{0,n}^J \right\|_{L^{2^\ast} (\Rm^d; \phi dx)} = o(J) \rightarrow 0,\; \hbox{as}\; J \rightarrow 0.
\]
By the identity $u_{0,n} = \sum_{j=1}^J V_{j,n} (\cdot, 0) + w_{0,n}^J$ and the limits above, we have 
\begin{align*}
 \limsup_{n \rightarrow \infty} \|u_{0,n}\|_{L^{2^\ast} (\Rm^d; \phi dx)} - o(J)  \leq  & \left(\sum_{j=1}^J   \lim_{n \rightarrow \infty} \int_{\Rm^d} \phi \left| V_{j,n} (x,0) \right|^{2^\ast} dx\right)^{1/2^{\ast}}\\ &\qquad\qquad \leq  \liminf_{n \rightarrow \infty} \|u_{0,n}\|_{L^{2^\ast} (\Rm^d; \phi dx)} + o(J)
\end{align*}
Letting $J \rightarrow \infty$, we finish the proof. 
\end{proof} 

In order to further simplify our profile decomposition, we will show that some of the profiles can be absorbed by the remainder $(w_{0,n}^J, w_{1,n}^J)$. The profiles that can be absorbed are identified by the following lemma. 

\begin{lemma}  \label{absorbable profiles}
 If the triple $(\lambda_j, x_j, t_j)$ satisfies one of the following
 \begin{itemize}
  \item $\lambda_j = \infty$;
  \item $\lambda_j < \infty$ and $x_j = \infty$;
 \end{itemize}
 then we have $\displaystyle \lim_{n \rightarrow \infty} \|V_{j,n}\|_{Y^\star (\Rm)} = 0$.
\end{lemma}

\begin{proof}
If the lemma failed, we would have a sequence $n_1<n_2<\cdots <n_k<\cdots$ such that $\|V_{j,n}\|_{Y^\star (\Rm)}> \eps_0$ for a fixed positive constant $\eps_0$. By possibly passing to a subsequence, we can assume $\lambda_{j,n_k} > 2^k$ if $\lambda_j = \infty$. We can calculate  
\begin{align*}
 \|V_{j,n}\|_{Y^\star(\Rm)}^{p_c} = &\int_{\Rm} \left[\int_{\Rm^d} |\phi(x)|^2\cdot \left|\frac{1}{\lambda_{j,n}^{\frac{d-2}{2}}} V_j \left(\frac{x-x_{j,n}}{\lambda_{j,n}},\frac{t -t_{j,n}}{\lambda_{j,n}}\right)\right|^{2p_c} dx\right]^{1/2} dt\\
 = &\int_{\Rm} \left[\int_{\Rm^d} |\phi (\lambda_{j,n} x + x_{j,n})|^2 \left| V_j \left(x,t\right)\right|^{2p_c} dx\right]^{1/2} dt
\end{align*}
Our assumptions guarantee that $\phi(\lambda_{j,n_k} x + x_{j,n_k})$ converges to zero almost everywhere in $\Rm^d$. Applying the dominated convergence theorem, we know that the integral above converges to zero as $k \rightarrow \infty$. This is a contradiction.
\end{proof}
\begin{remark}
 The argument above works for any energy-critical Strichartz norm $L^q L^r (\Rm \times \Rm^d)$ with $q, r< \infty$ as well. 
\end{remark}

Now we can substitute the profiles $V_j$'s that satisfy the conditions in Lemma \ref{absorbable profiles} ($\lambda_j = \infty$ or $x_j = \infty$) by zero, put the error into the remainder $(w_{0,n}^J, w_{1,n}^J)$ and obtain a modified version of the profile decomposition. Please note that those discarded profiles always satisfy the condition $\displaystyle \lim_{n \rightarrow \infty} \int_{\Rm^d} \phi |V_{j,n}(x,0)|^{2^\ast} dx = 0$, according to Lemma \ref{limit of phi V L6}. 

\begin{proposition}[Modified Profile Decomposition] \label{modified profile decomposition}
 Given a sequence $(u_{0,n}, u_{1,n}) \in H$ so that $\|(u_{0,n}, u_{1,n})\|_H \leq A$, there exist a subsequence of $(u_{0,n}, u_{1,n})$ (We still use the notation $(u_{0,n}, u_{1,n})$ for the subsequence), a sequence of free waves $V_j(x,t) = S_L (t) (v_{0,j}, v_{1,j})$ and a family of triples $(\lambda_{j,n}, x_{j,n}, t_{j,n}) \in \Rm^+ \times \Rm^3 \times \Rm$, which are "almost orthogonal", i.e. we have 
\[
 \lim_{n \rightarrow +\infty} \left(\frac{\lambda_{j',n}}{\lambda_{j,n}}+ \frac{\lambda_{j,n}}{\lambda_{j',n}} + \frac{|x_{j,n} - x_{j',n}|}{\lambda_{j,n}}+ \frac{|t_{j,n} - t_{j',n}|}{\lambda_{j,n}}\right) = + \infty
\]
for $j \neq j'$; such that
\begin{itemize}
\item [(I)] We have the decomposition 
\[
 (u_{0,n}, u_{1,n})  = \sum_{j =1}^J (V_{j,n}(\cdot, 0), \partial_t V_{j,n}(\cdot, 0)) + (w_{0,n}^J, w_{1,n}^J)
\]
Here $V_{j,n}$ is another free wave derived from $V_j$:
\begin{align*}
(V_{j,n}(\cdot, t), \partial_t V_{j,n}(\cdot, t)) &= \left(\frac{1}{\lambda_{j,n}^{\frac{d-2}{2}}} V_j\left(\frac{x - x_{j,n}}{\lambda_{j,n}}, \frac{t-t_{j,n}}{\lambda_{j,n}}\right),
 \frac{1}{\lambda_{j,n}^{d/2}} \partial_t V_j\left(\frac{x-x_{j,n}}{\lambda_{j,n}}, \frac{t-t_{j,n}}{\lambda_{j,n}}\right)\right)\\
& = S_L (t -t_{j,n}) \Di_{\lambda_{j,n},  x_{j,n}} (v_{0,j}, v_{1,j});
\end{align*}
\item [(II)] We have the following limits as $n \rightarrow \infty$:
\begin{align*}
 &\lambda_{j,n} \rightarrow \lambda_j \in \{0,1\};& &x_{j,n} \rightarrow x_j \in \Rm^d;& &\frac{-t_{j,n}}{\lambda_{j,n}} \rightarrow t_j \in [-\infty,\infty].&
\end{align*}
In addition, if $\lambda_j = 1$, then we must have $x_j = 0$. 
\item [(III)] The remainders $(w_{0,n}^J, w_{1,n}^J)$ satisfy:
\begin{align*}
  &\limsup_{n \rightarrow \infty} \|S_L (t)(w_{0,n}^J, w_{1,n}^J)\|_{Y^\star (\Rm)}\rightarrow 0,\; \hbox{as}\; J \rightarrow \infty.\\
  &\limsup_{n \rightarrow \infty} \|(w_{0,n}^J, w_{1,n}^J)\|_H \leq 2A.
\end{align*}
\item[(IV)] We have the limits
\begin{align*}
 \sum_{j=1}^\infty \|V_j\|_H^2 = \sum_{j=1}^\infty \|(v_{0,j}, v_{1,j})\|_H^2 & \leq \liminf_{n \rightarrow \infty} \|(u_{0,n}, u_{1,n})\|_H^2\\
 \sum_{j=1}^\infty \lim_{n\rightarrow \infty}\int_{\Rm^d} \phi |V_{j,n}(x,0)|^{2^\ast} dx & = \lim_{n \rightarrow \infty} \int_{\Rm^d} \phi |u_{0,n}(x)|^{2^\ast} dx
\end{align*}
\end{itemize}
\end{proposition}

\subsection{Nonlinear Profiles}

Given any linear profile decomposition as in Proposition \ref{modified profile decomposition}, we can assign a nonlinear profile to each linear profile $V_j$. Let us start by introducing the definition of a nonlinear profile. 

\begin{definition} [A nonlinear profile]
 Fix $\tilde{\phi}$ to be either the function $\phi$ or a constant $c$. Let $V(x,t) = S_L (t) (v_0,v_1)$ be a free wave and $\tilde{t} \in [-\infty,\infty]$ be a time. We say that $U(x,t)$ is a nonlinear profile associated to $(V,\tilde{\phi}, \tilde{t})$ if $U(x,t)$ is a solution to the nonlinear wave equation 
\begin{equation} \label{nonlinear profile1}
 \partial_t^2 u - \Delta u = \tilde{\phi} F(u)
\end{equation}
with a maximal timespan $I$ so that $I$ contains a neighbourhood\footnote{A neighbourhood of infinity is $(M, \infty)$ or $(-\infty, M)$} of $\tilde{t}$ and
\[
 \lim_{t \rightarrow \tilde{t}} \|(U(\cdot,t),\partial_t U (\cdot,t))- (V(\cdot,t), \partial_t V(\cdot,t))\|_H = 0.
\]
\end{definition}
\begin{remark} 
Given a triple $(V, \tilde{\phi}, \tilde{t})$ as above, one can show there is always a unique nonlinear profile. Please see Remark 2.13 in \cite{kenig1} for the idea of proof. In particular, if $\tilde{t}$ is finite, then the nonlinear profile is simply the solution to the equation \eqref{nonlinear profile1} with initial data $(V(\cdot, \tilde{t}), \partial_t V(\cdot, \tilde{t}))$ at the time $t = \tilde{t}$. We will also use the fact that the nonlinear profile automatically scatters in the positive time direction if $\tilde{t} = + \infty$. 
\end{remark}


\begin{definition} [Nonlinear Profiles]
 For each linear profile $V_j$ in a profile decomposition as given in Proposition \ref{modified profile decomposition}, we assign a nonlinear profile $U_j$ to it in the following way
 \begin{itemize}
 \item If $(\lambda_j,x_j) = (1,{\mathbf 0})$, then $U_j$ is chosen as the nonlinear profile associated to $(V_j, \phi, t_j)$.
 \item If $\lambda_j = 0$, then $U_j$ is chosen as nonlinear profile associated to $(V_j, \phi(x_j), t_j)$.
 \end{itemize}
In either case, we define 
\[
 U_{j,n}(x,t) \doteq \frac{1}{\lambda_{j,n}^{\frac{d-2}{2}}} U_j \left(\frac{x-x_{j,n}}{\lambda_{j,n}}, \frac{t - t_{j,n}}{\lambda_{j,n}}\right). 
\]
\end{definition}

\begin{lemma} [$U_{j,n}$ is an approximation solution to (CP1)] \label{approximation solution 1}
 If $I'_j$ is a time interval so that $\|U_j\|_{Y(I'_j)} < \infty$, then the error term 
\[
 e_{j,n} =  (\partial_t^2 - \Delta) U_{j,n} - \phi F(U_{j,n})
\]
satisfies the estimate 
\[
 \lim_{n \rightarrow \infty} \|e_{j,n}\|_{L^1 L^2 ((t_{j,n} + \lambda_{j,n} I'_j)\times \Rm^d)}  = 0. 
\]
\end{lemma}
\begin{proof}
If $(\lambda_{j,n}, x_{j,n})  \rightarrow (1,{\mathbf 0})$. We know $U_j$ satisfies 
\[
 \partial_t^2 U_j  -\Delta U_j = \phi(x) F(U_j) \; \Longrightarrow (\partial_t^2 - \Delta) U_{j,n} = \phi \left(\frac{x - x_{j,n}}{\lambda_{j,n}}\right) F(U_{j,n})
\]
Thus we have 
\begin{align*}
 \|e_{j,n}\|_{L^1 L^2 ((t_{j,n} + \lambda_{j,n} I'_j)\times \Rm^d)} = &\left\|\left(\phi \left(\frac{x - x_{j,n}}{\lambda_{j,n}}\right) -\phi(x)\right) F(U_{n,j})\right\|_{L^1 L^2 ((t_{j,n} + \lambda_{j,n} I'_j)\times \Rm^d)}\\
 = & \left\|\left(\phi \left(x\right) -\phi(\lambda_{j,n} x + x_{j,n})\right) F(U_j)\right\|_{L^1 L^2 (I'_j \times \Rm^d)} \rightarrow 0 
\end{align*}
Here we use the point-wise convergence $\phi(\lambda_{j,n} x + x_{j,n}) \rightarrow \phi(x)$ and the dominated convergence theorem. Now we consider the case $(\lambda_{j,n}, x_{j,n}) \rightarrow (0,x_j) $. Since $U_{j}$ is a solution to the equation $\partial_t^2 u -\Delta u = \phi (x_j) F(u)$, we know $U_{j,n}$ is a solution to the same equation. This implies  $e_{j,n} = \phi(x_j) F(U_{j,n}) - \phi F(U_{j,n})$. As a result we have 
\begin{align*}
 \|e_{j,n}\|_{L^1 L^2 ((t_{j,n} + \lambda_{j,n} I'_j)\times \Rm^d)}  = & \|(\phi(x_j)-\phi (x)) F(U_{j,n})\|_{L^1 L^2 ((t_{j,n} + \lambda_{j,n} I'_j)\times \Rm^d)}\\
 =  & \| (\phi(x_j) - \phi(\lambda_{j,n} x + x_{j,n})) F(U_j)\|_{L^1 L^2 (I'_j \times \Rm^d)} \rightarrow 0
\end{align*}
by the point-wise convergence $\phi(\lambda_{j,n} x + x_{j,n}) \rightarrow \phi(x_j)$ and the dominated convergence theorem. 
\end{proof}

We also need the following lemmas in order to find an upper bound for the $Y$ norm of $\sum_{j} U_{j,n}$.

\begin{lemma}[Almost Orthogonality of $U_{j,n}$] \label{general almost orthogonality}
Assume $\|\tilde{U}_j\|_{L^q L^r (I'_j \times \Rm^d)} < \infty$ for $j=1,2$. Here $L^q L^r$ is a Strichartz norm with $q < \infty$, i.e. we have  
\begin{align*}
 &2 \leq q, r < \infty;& &\frac{1}{q} + \frac{d}{r} = \frac{d}{2} - 1.&
\end{align*}
 Let $\{(\lambda_{1,n}, x_{1,n}, t_{1,n})\}_{n \in {\mathbb Z}^+}$ and $\{(\lambda_{2,n}, x_{2,n}, t_{2,n})\}_{n \in {\mathbb Z}^+}$  be two ``almost orthogonal'' sequences of triples, i.e. 
\[
 \lim_{n \rightarrow +\infty} \left(\frac{\lambda_{2,n}}{\lambda_{1,n}}+ \frac{\lambda_{1,n}}{\lambda_{2,n}} + \frac{|x_{1,n} - x_{2,n}|}{\lambda_{1,n}}+ \frac{|t_{1,n} - t_{2,n}|}{\lambda_{1,n}}\right) = + \infty.
\]    
If $I_n$ is a sequence of time intervals, such that $I_n \subseteq (t_{1,n} + \lambda_{1,n}I'_1)\cap (t_{2,n} + \lambda_{2,n}I'_2)$, then we have 
\[
 N(n) = \left\|\tilde{U}_{1,n} \tilde{U}_{2,n} \right\|_{L_t^{q/2} L_x^{r/2} (I_n \times \Rm^d)} \rightarrow 0, \qquad \hbox{as}\;\; n \rightarrow \infty. 
\]
Here $\tilde{U}_{j,n}$ is defined as usual 
\[
 \tilde{U}_{j,n} = \frac{1}{\lambda_{j,n}^{\frac{d-2}{2}}} \tilde{U}_j \left(\frac{x-x_{j,n}}{\lambda_{j,n}}, \frac{t - t_{j,n}}{\lambda_{j,n}}\right) 
\]
\end{lemma}
\begin{proof} (See also Lemma 2.7 in \cite{orthogonality}) First of all, by defining $\tilde{U}_j (x,t) = 0$ for $t \notin I'_j$, we can always assume $I'_j = \Rm$ and $I_n = \Rm$. Observing the continuity of the map 
\[
 \Phi: L^q L^r (\Rm \times \Rm^d) \times L^q L^r (\Rm \times \Rm^d) \rightarrow l^\infty,\quad \Phi (\tilde{U}_1,\tilde{U}_2) = \left\{\left\|\tilde{U}_{1,n} \tilde{U}_{2,n}\right\|_{L^{q/2} L^{r/2}(\Rm \times \Rm^d)}\right\}_{n \in {\mathbb Z}^+},
\]
we can also assume, without loss of generality, that 
\begin{align*}
 &|\tilde{U}_j (x,t)| \leq M_j, \; \hbox{for any}\; (x,t)\in \Rm^d \times \Rm;& &\hbox{Supp} (\tilde{U}_j) \subseteq \{(x,t): |x|,|t| < R_j\} &
\end{align*}
for each $j = 1,2$ and some constants $M_j$ and $R_j$, since the functions satisfying these conditions are dense in the space $L^q L^r (\Rm \times \Rm^d)$.  If the conclusion were false, we would find a sequence $n_1 < n_2 < n_3 < \cdots$ and a positive constant $\eps_0$ such that $ N(n_k) \geq \eps_0$. There are three cases
\begin{itemize}
 \item [(I)] $\limsup_{k \rightarrow \infty} \lambda_{1, n_k}/ \lambda_{2,n_k} = \infty$. In this case the product $\tilde{U}_{1,n_k} \tilde{U}_{2,n_k}$ is supported in the $(d+1)$-dimensional circular cylinder centred at $(x_{2,n_k}, t_{2,n_k})$ with radius $\lambda_{2,n_k} R_2$ and height $2\lambda_{2,n_k} R_2$ since $\tilde{U}_{2,n_k}$ is supported in this cube. On the other hand, we have 
 \[
  \left|\tilde{U}_{1,n_k} \tilde{U}_{2,n_k}\right| \leq \lambda_{1,n_k}^{-\frac{d-2}{2}} \lambda_{2,n_k}^{-\frac{d-2}{2}} M_1 M_2.
 \] 
A basic computation shows 
\[
N(n_k) = \left\|\tilde{U}_{1,n_k} \tilde{U}_{2,n_k} \right\|_{L^{q/2} L^{r/2} (\Rm \times \Rm^d)} \leq C(d) M_1 M_2 R_2^{d-2} \left(\frac{\lambda_{2, n_k}}{\lambda_{1,n_k}}\right)^{\frac{d-2}{2}}. 
\]
The upper bound tends to zero as $\lambda_{1, n_k}/ \lambda_{2,n_k} \rightarrow \infty$. Thus we have a contradiction. 
\item [(II)] $\limsup_{k \rightarrow \infty} \lambda_{2, n_k}/ \lambda_{1,n_k} = \infty$. This can be handled in the same way as case (I).
\item [(III)] $\lambda_{1, n_k} \simeq \lambda_{2, n_k}$. By the ``almost orthogonality'' of the sequences of triples, we also have 
\[
 \frac{|x_{1,n_k} - x_{2,n_k}|}{\lambda_{1,n_k}}+ \frac{|t_{1,n_k} - t_{2,n_k}|}{\lambda_{1,n_k}} \rightarrow \infty.
\]
This implies $\hbox{Supp}(\tilde{U}_{1,n_k}) \cap \hbox{Supp}(\tilde{U}_{2,n_k}) = \emptyset$ when $k$ is sufficiently large thus gives a contradiction. 
\end{itemize}
\end{proof}

\begin{lemma} \label{uniform estimate in J}
Let $I'_j$ be a time interval such that $\|U_j(x,t)\|_{Y(I'_j)} < \infty$. Suppose $\{J_n\}$ is a sequence of time intervals, so that $\displaystyle J_n \subseteq \cap_{j =1}^J (t_{j,n} + \lambda_{j,n} I'_j)$ holds for sufficiently large $n$. Then we have
\begin{align*}
 &\lim_{n \rightarrow \infty} \left\|F\left(\sum_{j=1}^J U_{j,n}\right) - \sum_{j=1}^J F(U_{j,n})\right\|_{L^1 L^2 (J_n \times \Rm^d)} = 0.\\
 &\limsup_{n \rightarrow \infty} \left\|\sum_{j=1}^J U_{j,n} \right\|_{Y(J_n)} \leq \left(\sum_{j=1}^{J}\|U_j\|_{Y(I'_j)}^{p_c}\right)^{1/p_c}.
\end{align*}
\end{lemma}
\paragraph{Proof} First of all, for any $J \geq 2$ we have 
\begin{align*}
& \lim_{n \rightarrow \infty} \left\|F\left(\sum_{j=1}^J U_{j,n}\right) - F\left(\sum_{j=1}^{J-1} U_{j,n}\right) - F(U_{J,n})\right\|_{L^1 L^2 (J_n \times \Rm^d)}\\
=& \lim_{n \rightarrow \infty} \left\|\left[U_{J,n} \int_0^1 F'\left(\tau U_{J,n} + \sum_{j=1}^{J-1} U_{j,n}\right) d\tau \right]
- \left[U_{J,n} \int_0^1 F'(\tau U_{J,n})d\tau\right] \right\|_{L^1 L^2 (J_n \times \Rm^d)}\\
=& \lim_{n \rightarrow \infty} \left\| \left(U_{J,n} \sum_{j=1}^{J-1} U_{j,n}\right) \left( \int_0^1 \int_0^1
F''\left(\tau U_{J,n} + \tilde{\tau}\sum_{j=1}^{J-1} U_{j,n}\right) d\tilde{\tau} d\tau\right)\right\|_{L^1 L^2 (J_n \times \Rm^d)}\\
\leq& \lim_{n \rightarrow \infty} \left(\sum_{j=1}^{J-1} \|U_{J,n} U_{j,n}\|_{L^{p_c/2} L^{p_c} (J_n \times \Rm^d)}\right)
\left( \sum_{j=1}^{J}\|U_{j,n}\|_{Y(J_n)}\right)^{p_c -2}\\
\leq& \lim_{n \rightarrow \infty}  \left(\sum_{j=1}^{J-1} \|U_{J,n} U_{j,n}\|_{L^{p_c/2} L^{p_c} (J_n \times \Rm^d)}\right)
\left( \sum_{j=1}^{J}\|U_j\|_{Y(I'_j)}\right)^{p_c -2}\\
=& 0.
\end{align*}
In the final step we apply Lemma \ref{general almost orthogonality}. As a result we can prove the first limit by an induction. The second limit is a corollary. 
\begin{align*}
 \limsup_{n\rightarrow \infty} \left\|\sum_{j=1}^J U_{j,n}\right\|_{Y(J_n)}^{p_c}
 =& \limsup_{n \rightarrow \infty} \left\|F\left(\sum_{j=1}^J U_{j,n}\right)\right\|_{L^1 L^2 (J_n \times \Rm^d)}\\
 = & \limsup_{n \rightarrow \infty} \left\|\sum_{j=1}^J F(U_{j,n})\right\|_{L^1 L^2 (J_n \times \Rm^d)} \\
 \leq& \limsup_{n \rightarrow \infty} \left(\sum_{j=1}^{J}\|F(U_{j,n})\|_{L^1 L^2 (J_n \times \Rm^d)}\right)\\
 = & \limsup_{n \rightarrow \infty} \left(\sum_{j=1}^{J}\|U_{j,n}\|_{Y(J_n)}^{p_c}\right)\\
 \leq& \sum_{j=1}^{J}\|U_{j}\|_{Y(I'_j)}^{p_c}.
\end{align*}

\begin{lemma} [Energy of a Nonlinear Profile] \label{the energy of Uj}
Let $U_j$ be the nonlinear profile as above. Then we have the energy defined by
\[
 E(U_j) = \left\{\begin{array}{ll} \int_{\Rm^d} \left(\frac{1}{2}|\nabla U_j(x,t)|^2 + \frac{1}{2}|\partial_t U_j(x,t)|^2 - \frac{\zeta \phi}{2^\ast} |U_j(x,t)|^{2^\ast}  \right) dx, & \hbox{if}\; \lambda_j = 1;\\
  \int_{\Rm^d} \left(\frac{1}{2}|\nabla U_j(x,t)|^2 + \frac{1}{2} |\partial_t U_j(x,t)|^2 - \frac{\zeta \phi(x_j)}{2^\ast} |U_j(x,t)|^{2^\ast}  \right) dx, & \hbox{if}\; \lambda_j = 0;
 \end{array}\right.
\]
satisfies 
\[
 E(U_j) = \frac{1}{2}\|V_j\|_H^2 - \frac{\zeta}{2^\ast} \lim_{n \rightarrow \infty} \int_{\Rm^d} \phi |V_{j,n} (x,0)|^{2^\ast} dx.  
\]
\end{lemma}
\begin{remark}
According to our definition of nonlinear profiles, the energy defined above does not depend on the time $t$. 
\end{remark}
\begin{proof} There are two cases. 
\paragraph{Case 1} If $t_j = +\infty$, then the definition of the nonlinear profile gives 
\begin{equation} \label{nonlinear profile limit infinity}
 \lim_{t\rightarrow +\infty} \left\|\begin{pmatrix} U_j (\cdot,t)\\ \partial_t U_j (\cdot,t)\end{pmatrix} - \begin{pmatrix} V_j (\cdot,t)\\ \partial_t V_j (\cdot,t)\end{pmatrix}\right\|_{\dot{H}^1 \times L^2 (\Rm^d)} = 0.
\end{equation}
By the Sobolev embedding, we also have $\|U_j(\cdot,t)- V_j (\cdot, t)\|_{L^{2^\ast}} \rightarrow 0$ as $t \rightarrow + \infty$. Applying Lemma \ref{L6 tends to zero}, we obtain
\[
 \lim_{t \rightarrow +\infty} \int_{\Rm^d} |U_j (x,t)|^{2^\ast} dx = 0. 
\]
Combining this with \eqref{nonlinear profile limit infinity}, we can evaluate the energy at larger and larger times $t$ and obtain $E(U_j) = \frac{1}{2} \|V_j\|_H^2$. This finishes the proof, since we also know 
\[
 \lim_{n \rightarrow \infty} \int_{\Rm^d}  |V_{j,n} (x,0)|^{2^\ast} dx = \lim_{n \rightarrow \infty} \int_{\Rm^d} |V_j  (x, -t_{j,n}/\lambda_{j,n})|^{2^\ast} dx = 0.
\]
The same argument works if $t_j = - \infty$. 
\paragraph{Case 2} If $t_j$ is finite, we can immediately conclude the proof if we evaluate the energy at $t = t_j$ by using Lemma \ref{limit of phi V L6} and the fact $(U_j(\cdot, t_j), \partial_t U_j(\cdot,t_j)) = (V_j(\cdot, t_j), \partial_t V_j (\cdot, t_j))$. 
\end{proof}
Combining Lemma \ref{the energy of Uj} and part (IV) of our modified profile decomposition (Proposition \ref{modified profile decomposition}), we obtain 
\begin{corollary} \label{sum of energy}
 Let $U_j$'s be the nonlinear profiles with energy $E(U_j)$ as defined above. Then we have the inequality 
\[
 \sum_{j=1}^\infty E(U_j) \leq \liminf_{n \rightarrow \infty} E_\phi (u_{0,n}, u_{1,n}). 
\]
\end{corollary}

\section{Compactness Procedure}

In this section we prove the following proposition
\begin{proposition} \label{compactness process}
Assume that $\phi$ satisfies the condition \eqref{basic condition of phi}. If the statement SC($\phi$, M) breaks down at $M_0 < E_1 (W,0)$, i.e. the statement holds for all $M \leq M_0$ but fails for any $M > M_0$, then there exists a critical element $u$, which is a solution to (CP1) in the focusing case with initial data $(u_0,u_1)$ such that
 \begin{itemize}
  \item The energy $E_\phi (u_0,u_1) = M_0$; 
  \item The solution exists globally in time with $\|u\|_{Y^\star ([0,\infty))} = \|u\|_{Y^\star ((-\infty,0])} = + \infty$. 
  \item The norm $\|(u(\cdot, t), u_t (\cdot,t))\|_{\dot{H}^1 \times L^2 (\Rm^d)} < \|\nabla W\|_{L^2}$ for each $t \in \Rm$.
  \item The set $\{(u(t), \partial_t u(t))| t\in \Rm\}$ is pre-compact in $\dot{H}^1 \times L^2 (\Rm^d)$. 
 \end{itemize}
\end{proposition}

\begin{remark}
The compactness procedure in the defocusing case is similar. We can substitute the statement SC($\phi$, M) and Proposition \ref{compactness process} by the statement SC'($\phi$, M) and Proposition \ref{compactness process defocusing} as below.
\end{remark}

\begin{statement}[SC'($\phi$, M)]
 There exists a function $\beta: [0,M) \rightarrow \Rm^+$, such that if the initial data $(u_0,u_1) \in \dot{H}^{1} \times L^2 (\Rm^d)$ satisfy
 $E_\phi (u_0,u_1) < M$, then the corresponding solution $u$ to (CP1) in the defocusing case exists globally in time, scatters in both two time directions with 
 \[
  \|u\|_{Y(\Rm)} < \beta (E_\phi (u_0,u_1)).
 \]
\end{statement}

\begin{proposition} \label{compactness process defocusing}
Assume that $\phi$ satisfies the condition \eqref{basic condition of phi}. If the statement SC'($\phi$, M) breaks down at $M_0$, i.e. the statement holds for all $M \leq M_0$ but fails for any $M > M_0$, then there exists a critical element $u$, which is a solution to (CP1) in the defocusing case with initial data $(u_0,u_1)$ such that
 \begin{itemize}
  \item The energy $E_\phi (u_0,u_1) = M_0$; 
  \item The solution exists globally in time with $\|u\|_{Y^\star ([0,\infty))} = \|u\|_{Y^\star ((-\infty,0])} = + \infty$. 
  \item The set $\{(u(t), \partial_t u(t))| t\in \Rm\}$ is pre-compact in $\dot{H}^1 \times L^2 (\Rm^d)$. 
 \end{itemize}
\end{proposition}

\subsection{Set-up of the Proof} \label{sec: beginning of compactness}
If the statement SC($\phi$, M) breaks down at $M_0 < E_1 (W,0)$, then the statement $SC(\phi, M_0 +2^{-n})$ is not true for each positive integer $n$. Fix a positive integer $N_0$ so that $M_0 + 2^{-n} < E_1(W,0)$ for each $n \geq N_0$.  Under these assumptions we have 
\begin{lemma} \label{starting sequence}
We can find a sequence of solutions $\{u_n\}_{n \geq N_0}$ with initial data $(u_{0,n}, u_{1,n}) \in \dot{H}^1 \times L^2 (\Rm^d)$, such that 
\begin{itemize}
 \item[(I)] $\|\nabla u_{0,n}\|_{L^2} < \|\nabla W\|_{L^2}$ and $E_\phi (u_{0,n}, u_{1,n}) < M_0 + 2^{-n}$;
 \item[(II)] Let $(-T_-(u_{0,n}, u_{1,n}), T_+ (u_{0,n}, u_{1,n}))$ be the maximal lifespan of $u_n$. We have 
 \begin{align*}
  &\|u_n\|_{Y^\star ((-T_- (u_{0,n},u_{1,n}), 0])} > 2^n;& &\|u_n\|_{Y^\star ([0,T_+(u_{0,n}, u_{1,n})))} > 2^n.&
 \end{align*}
\end{itemize} 
\end{lemma}
\begin{proof}
Given $n \geq N_0$, let us first show there exists a solution $v_n$ with initial data $(v_{0,n},v_{1,n})$ and maximal lifespan $(-T_-,T_+)$, so that part (I) of the conclusion above holds and 
\[
 \|v_n\|_{Y^\star ((-T_-,T_+))} > 2^{n+\frac{1}{p_c}}.
\] 
If this were false, then the statement $SC(M_0+ 2^{-n})$ would be true, because we could find a universal upper bound for $Y$ norm as well, according to Remark \ref{Y Ystar equivalent} and Remark \ref{positive energy}. Next we can pick a time $t_0 \in (-T_-,T_+)$, so that $\|v_n\|_{Y^\star ([t_0,T_+))} > 2^n$ and $\|v_n\|_{Y^\star ((-T_-,t_0])} > 2^n$. Finally we finish the proof by choosing 
\begin{align*}
 &(u_{0,n},u_{1,n}) = (v_n(\cdot, t_0), \partial_t v_n (\cdot,t_0));& &u_n(\cdot,t) = v_n (\cdot,t+t_0).&
\end{align*}
Note that the conservation law of energy and Lemma \ref{energy trapping} guarantee the new initial data $(u_{0,n},u_{1,n})$ still satisfy (I).
\end{proof}

\paragraph{Application of the profile decomposition} Let us consider the solutions $u_n$ and initial data $(u_{0,n}, u_{1,n})$ given above. According to Lemma \ref{energy trapping}, there exists a constant $0 <\bar{\delta} < 1$, such that
\[
 \|(u_{0,n},u_{1,n})\|_{\dot{H}^1 \times L^2} < (1 - \bar{\delta}) \|\nabla W\|_{L^2} 
\]
holds for large $n$. Thus we are able to apply the linear profile decomposition (Proposition \ref{modified profile decomposition}) on the sequence $\{(u_{0,n}, u_{1,n})\}_{n \in {\mathbb Z}^+}$, then assign a nonlinear profile $U_j$ to each linear profile $V_j$ as we did in the previous section. Finally we obtain the decomposition ($J \in {\mathbb Z}^+$)
\begin{align}
 (u_{0,n}, u_{1,n}) = & \sum_{j=1}^J (V_{j,n}(\cdot, 0), \partial_t V_{j,n}(\cdot, 0)) + (w_{0,n}^J, w_{1,n}^J)\nonumber \\
 = & (S_{J,n}(\cdot,0), \partial_t S_{J,n}(\cdot, 0)) + (\tilde{w}_{0,n}^J, \tilde{w}_{1,n}^J) 
\end{align}
Here $\{(u_{0,n}, u_{1,n})\}_{n \in {\mathbb Z}^+}$ is actually a subsequence of the original sequence of initial data, although we still use the same notation. One can check that the conclusion of Lemma \ref{starting sequence} still holds for this subsequence (along with the corresponding solutions $u_n$). The notations $S_{J,n}$ and $(\tilde{w}_{0,n}^J, \tilde{w}_{1,n}^J)$ represent 
\begin{align*}
 S_{J,n} &= \sum_{j=1}^J U_{j,n};\\
 (\tilde{w}_{0,n}^J, \tilde{w}_{1,n}^J) &= (w_{0,n}^J, w_{1,n}^J) + \sum_{j=1}^J \left(V_{j,n}(\cdot, 0) - U_{j,n}(\cdot, 0), \partial_t V_{j,n}(\cdot, 0) - \partial_t U_{j,n}(\cdot, 0)\right).
\end{align*}
By our definition of nonlinear profiles, we have ($j \in \mathbb{Z}^+$)
\[
 \lim_{n \rightarrow \infty} \left\|\left(V_{j,n}(\cdot, 0) - U_{j,n}(\cdot, 0), \partial_t V_{j,n}(\cdot, 0) - \partial_t U_{j,n}(\cdot, 0)\right)\right\|_{\dot{H}^1 \times L^2} = 0.
\]
Thus we still have 
\begin{align} \label{convergence of tilde w}
 &\limsup_{n \rightarrow \infty} \left\|S_L (t)(\tilde{w}_{0,n}^J, \tilde{w}_{1,n}^J)\right\|_{Y^\star (\Rm)} \rightarrow 0,\qquad \hbox{as}\;\; J \rightarrow 0;&\\
 &\limsup_{n \rightarrow \infty} \|(\tilde{w}_{0,n}^J, \tilde{w}_{1,n}^J)\|_H < 2 \|\nabla W\|_{L^2}.&
\end{align}
In addition, part (IV) of the conclusion in the profile decomposition gives
\begin{equation} \label{sum of E0}
 \sum_{j=1}^\infty \|V_j\|_H^2 \leq \liminf_{n \rightarrow \infty} \|(u_{0,n}, u_{1,n})\|_H^2 \leq (1 - \bar{\delta})^2 \|\nabla W\|_{L^2}^2.
\end{equation} 
By the definition of nonlinear profiles, we also have 
 \begin{equation} \label{UV are close in H}
 \lim_{t \rightarrow t_j} \|(U_j(\cdot, t), \partial_t U_j (\cdot, t))\|_H = \|V_j\|_H.
 \end{equation}
Combining \eqref{sum of E0} and \eqref{UV are close in H}, we obtain
\begin{itemize}
 \item For each given $j$, we have
 \begin{equation}
 \lim_{t \rightarrow t_j} \|(U_j(\cdot, t), \partial_t U_j (\cdot, t))\|_H  \leq (1 -\bar{\delta}) \|\nabla W\|_{L^2} \label{H upper bound on Uj}
 \end{equation}
According to Remark \ref{positive energy}, we have $E(U_j) > 0$ unless $U_j$ is identically zero. By Corollary \ref{sum of energy} we also have $\displaystyle \sum_{j=1}^\infty E(U_j) \leq M_0$. 
\item As $j \rightarrow \infty$, we have
\[
   \lim_{t \rightarrow t_j} \|(U_j(\cdot, t), \partial_t U_j (\cdot, t))\|_H = \|V_j\|_H \rightarrow 0.
\]
Thus we know $U_j$ is globally defined in time and scatters with $\|U_j\|_{Y(\Rm)} \lesssim \|V_j\|_H$ for each sufficiently large $j > J_0$. Combining this upper bound on $\|U_j\|_{Y(\Rm)}$ and the inequality \eqref{sum of E0}, we have
\begin{equation} \label{Y norm of sum S}
 \sum_{j = J_0+1}^\infty \|U_j\|_{Y(\Rm)}^{p_c} < \infty.
\end{equation}
\end{itemize}

\subsection{The Extraction of a Critical Element}

In this subsection, we show there is exactly one nonzero profile in the profile decomposition, whose corresponding nonlinear profile is exactly the critical element we are looking for. We start by proving 
\begin{lemma} \label{positive time direction 1}
 If each nonlinear profile $U_j$ we obtained in the previous subsection scatters in the positive time direction, then $u_n$ scatters in the positive time direction for sufficiently large $n > N_0$ and 
 \[
  \sup_{n > N_0} \|u_n\|_{Y^\star ([0,\infty))} < \infty.
 \]
\end{lemma}
\begin{proof}
 Let us consider the approximate solution $S_{J,n} = \sum_{j=1}^J U_{j,n}$ which satisfies the equation 
\[
 (\partial_t^2 - \Delta) S_{J,n} - \phi F(S_{J,n}) = e'_{J,n}, \qquad t \in [0,\infty)
\]
with the error term 
\[
 e'_{J,n} = \phi \left[\sum_{j=1}^J F(U_{j,n}) - F(S_{J,n})\right] + \sum_{j=1}^J \left[(\partial_t^2 - \Delta) U_{j,n} - \phi F(U_{j,n})\right]
\]
and the initial data $(n_{0,n}, u_{1,n}) = (S_{J,n}(\cdot,0), \partial_t S_{J,n}(\cdot, 0)) + (\tilde{w}_{0,n}^J, \tilde{w}_{1,n}^J)$. We use the notation $I_j$ for the maximal lifespan of $U_j$. By our assumption on scattering we can choose an interval $I'_j \subseteq I_j$ for each $j$ as below so that $\|U_j\|_{Y(I'_j)} < \infty$, 
\[
  I'_j = \left\{\begin{array}{ll} (-\infty,\infty) & \hbox{if}\; j > J_0\; \hbox{or}\; t_j = - \infty;\cr
  [t_j^-, \infty) & \hbox{if}\; j \leq J_0\; \hbox{and}\; t_j > -\infty, \; \hbox{here $t_j^- \in I_j$ is a time smaller than $t_j$}.
  \end{array}\right.
\]
One can check that $[0,\infty) \subseteq \lambda_{j,n} I'_j + t_{j,n}$ holds for all $j \in {\mathbb Z}^+$ as long as $n$ is sufficiently large. Thus we can apply Lemma \ref{approximation solution 1} as well as Lemma \ref{uniform estimate in J} and obtain for each $J$
\begin{align}
 & \lim_{n \rightarrow \infty} \|e'_{J,n}\|_{L^1 L^2 ([0,\infty) \times \Rm^d)} = 0; \label{upper bound for error compactness} \\
 & \limsup_{n \rightarrow \infty} \left\|S_{J,n}\right\|_{Y([0,\infty))} \leq \left(\sum_{j=1}^{J}\|U_j\|_{Y(I'_j)}^{p_c}\right)^{1/p_c} 
 \leq \left(\sum_{j=1}^{\infty}\|U_j\|_{Y(I'_j)}^{p_c}\right)^{1/p_c} < \infty. \label{upper bound for Y compactness}
\end{align}
Here we use our estimate \eqref{Y norm of sum S}. Let 
\[
 M_1 = \left(\sum_{j=1}^{\infty}\|U_j\|_{Y(I'_j)}^{p_c}\right)^{1/p_c} +1
\]
and $\eps_0 = \eps_0 (M_1)$ be the constant given in the long-time perturbation theory (Theorem \ref{perturbation theory in Y star}). Let us first fix a $J_1$ so that $\displaystyle \limsup_{n \rightarrow \infty} \|S_L (t) (\tilde{w}_{0,n}^{J_1}, \tilde{w}_{1,n}^{J_1})\|_{Y^\star (\Rm)} <  \eps_0/2$. Using this upper limit as well as \eqref{upper bound for error compactness} and \eqref{upper bound for Y compactness}, we can find a number $N_0 \in {\mathbb Z}^+$, such that if $n > N_0$, then 
\begin{align*}
  &\|e'_{J_1,n}\|_{L^1 L^2 ([0,\infty) \times \Rm^d)} < \eps_0/2;& & \|S_L (t) (\tilde{w}_{0,n}^{J_1}, \tilde{w}_{1,n}^{J_1})\|_{Y^\star (\Rm)} < \eps_0/2;&\\
  &\left\|S_{J_1,n}\right\|_{Y([0,\infty))} <  M_1 = \left(\sum_{j=1}^{\infty}\|U_j\|_{Y(I'_j)}^{p_c}\right)^{1/p_c} +1. &
\end{align*}
These estimates enable us to apply the long-time perturbation theory on the approximation solution $S_{J_1,n}$, initial data $(u_{0,n}, u_{1,n})$ and the time interval $[0,\infty)$, and then to conclude $u_n$ scatters in the positive time direction with 
\[
 \|u_n\|_{Y^\star ([0,\infty))} \leq \|S_{J_1,n}\|_{Y^\star ([0,\infty))}  + \|u_n - S_{J_1,n}\|_{Y^\star([0,\infty))} \leq M_1 + C(M_1) \eps_0 < \infty
\]
for each $n > N_0$. 
\end{proof}

\paragraph{Critical Element} Because we have assumed that $\|u_n\|_{Y^\star [0,T_+(u_{0,n}, u_{1,n}))} > 2^n$, Lemma \ref{positive time direction 1} implies that there is at least one nonlinear profile, say $U_{j_0}$, that fails to scatter in the positive time direction. In addition, the limit \eqref{H upper bound on Uj} implies $\|\nabla U_{j_0} (\cdot, t)\|_{L^2} < \|\nabla W\|_{L^2}$ when $t$ is close to $t_{j_0}$. According to our assumption that $SC(\phi, M)$ is true for any $M \leq M_0$ and Corollary \ref{critical focusing phi equal c}, we obtain
\[
 E(U_{j_0}) = \left\{\begin{array}{ll} \geq M_0, & \hbox{if}\; \lambda_{j_0} = 1;\\
 \geq E_1 (W,0) > M_0, & \hbox{if}\; \lambda_{j_0} = 0.
 \end{array}\right.
\]
Combining this with the already known facts that $\sum_{j=1}^\infty E(U_j) \leq M_0$ and $E(U_j) \geq 0$ (please see the bottom part of Subsection \ref{sec: beginning of compactness}), we obtain that 
\begin{itemize} 
 \item $U_{j_0}$ is a solution to (CP1) with an energy $E (U_{j_0}) = E_\phi (U_{j_0}, \partial_t U_{j_0}) = M_0$.
 \item $U_{j_0}$ is the only nonlinear profile with a positive energy;
 \item Any other profile $U_j$ is identically zero;
 \item By Lemma \ref{energy trapping}, the inequality $\|(U_{j_0} (\cdot, t), \partial_t U_{j_0} (\cdot, t))\|_H < \|\nabla W\|_{L^2}$ holds for all time $t$ in the maximal lifespan. 
\end{itemize}
\begin{remark} \label{negative time direction}
 A similar result as Lemma \ref{positive time direction 1} holds for the negative time direction, because the wave equation is time-invertible. This implies that $U_{j_0}$ fails to scatter in the negative time direction as well. A direct corollary follows that $t_{j_0}$ is finite. 
\end{remark}

\subsection{Almost Periodicity}

Let $u$ be the critical element $(U_{j_0})$ we obtained in the previous subsection and $I$ be its maximal lifespan. In this subsection we prove that 
\begin{proposition}
 The set $\{(u(\cdot, t), \partial_t u(\cdot,t))| t \in I\}$ is pre-compact in $\dot{H}^1 \times L^2 (\Rm^d)$.
\end{proposition}
\begin{proof}
Given an arbitrary sequence of time $\{t_n\}_{n \in {\mathbb Z}^+}$ so that $t_n \in I$, we know $u_n = u(\cdot, t + t_n)$ is still a solution to (CP1) with initial data $(u_{0,n}, u_{1,n})=  (u(\cdot, t_n), \partial_t u(\cdot, t_n))$ at the time $t =0$. This sequence of solutions still satisfies the conclusion of Lemma \ref{starting sequence}. Thus we can repeat the process we followed in the previous subsections. Finally we can find a subsequence $\{u_{n_k}\}_{k \in {\mathbb Z}^+}$ with a single linear profile $V_1$, a single nonlinear profile $U_1$ and a sequence of triples $(\lambda_{1,k}, x_{1,k}, t_{1,k})$ such that 
\begin{itemize}
 \item[(a)] $(n_{0,n_k}, u_{1,n_k}) = (U_{1,k} (\cdot,0), \partial_t U_{1,k}(\cdot, 0)) + (\tilde{w}_{0,k}, \tilde{w}_{1,k})$;
 \item[(b)] $\displaystyle \limsup_{k \rightarrow \infty} \|S_L (t) (\tilde{w}_{0,k}, \tilde{w}_{1,k})\|_{Y^\star(\Rm)} = 0$;
 \item[(c)] $\lambda_{1,k} \rightarrow 1$, $x_{1,k} \rightarrow 0$ and $-t_{1,k}/\lambda_{1,k} \rightarrow t_1 \in \Rm$;
 \item[(d)] $E_\phi (U_1, \partial_t U_1)= M_0$ and $\displaystyle \limsup_{k \rightarrow \infty} \|(\tilde{w}_{0,k}, \tilde{w}_{1,k})\|_H < 2 \|\nabla W\|_{L^2}$;
 \item[(e)] By the fact $U_1(\cdot, t_1) = V_1(\cdot, t_1)$, Lemma \ref{limit of phi V L6} and part (IV) of Proposition \ref{modified profile decomposition}, we have 
 \[
  \int_{\Rm^d} \phi(x) |U_1 (x, t_1)|^{2^\ast} dx = \lim_{k \rightarrow \infty} \int_{\Rm^d} \phi |u_{0,n_k}(x)|^{2^\ast} dx.
 \]
 \end{itemize}
Now let us prove that $(u_{0,n_k}, u_{1,n_k})$ converges to $ (U_1 (\cdot,t_1), \partial_t U_1(\cdot, t_1))$ strongly in $H$, as $k \rightarrow \infty$. First of all, the condition (c) above implies 
that 
\begin{equation}\label{Strong convergence of U1k}
 (U_{1,k} (\cdot,0), \partial_t U_{1,k}(\cdot, 0)) \rightarrow (U_1 (\cdot,t_1), \partial_t U_1(\cdot, t_1)) \; \hbox{strongly in}\; H
\end{equation}
Thus we can substitute Condition (a) above by 
\begin{equation} \label{new con a}
 (n_{0,n_k}, u_{1,n_k}) = (U_1 (\cdot,t_1), \partial_t U_1(\cdot, t_1)) + (\tilde{w}_{0,k}, \tilde{w}_{1,k})
\end{equation}
Here the remainders $(\tilde{w}_{0,k}, \tilde{w}_{1,k})$ may be different from the original ones, but they still satisfy the same estimates in (b) and (d). According to Lemma \ref{weak convergence of w}, we know $(\tilde{w}_{0,k}, \tilde{w}_{1,k})$ converges to zero weakly in $H$. Thus
\begin{align}
 & \lim_{k \rightarrow \infty} \left\langle (U_1 (\cdot,t_1), \partial_t U_1(\cdot, t_1)), (\tilde{w}_{0,k}, \tilde{w}_{1,k}) \right\rangle_H =0. \label{almost orthogonality wU1}
\end{align}
In addition, using Condition (e) and the fact $E_\phi (U_1(\cdot, t_1), \partial_t U_1 (\cdot, t_1)) = M_0 = E_\phi (n_{0,n_k}, u_{1,n_k})$, we obtain
\begin{equation} \label{H norm equivalent Uu}
\lim_{k \rightarrow \infty} \|(n_{0,n_k}, u_{1,n_k})\|_H^2  = \|(U_1(\cdot, t_1), \partial_t U_1 (\cdot, t_1))\|_H^2.
\end{equation}
Finally we can combine \eqref{new con a}, \eqref{almost orthogonality wU1} and \eqref{H norm equivalent Uu} to conclude that $\displaystyle \lim_{k \rightarrow \infty} \|(\tilde{w}_{0,k}, \tilde{w}_{1,k})\|_H \rightarrow 0$, which is equivalent to the strong convergence of $(n_{0,n_k}, u_{1,n_k})$, namely the strong convergence of $(u(\cdot, t_n), \partial_t u(\cdot, t_n))$, in light of \eqref{new con a}.
\end{proof}

\subsection{Global existence in time}

According to Remark \ref{lifespan lower bound for compact set}, the almost periodic property guarantees the existence of a constant $T>0$, such that if $t$ is contained in the maximal lifespan $I$ of $u$, then $(t-T,t+T) \subseteq I$. This implies $I = \Rm$.  Collecting all the properties of the critical element $u$ we obtained earlier, we can conclude the proof of Proposition \ref{compactness process}.

\section{Rigidity}

In this section we show that the critical element obtained in the previous section can never exist, thus finish the proof of the scattering part of our main theorem. There are two cases.

\begin{itemize}
 \item In the defocusing case, A Morawetz-type estimate is sufficient to finish the job. 
 \item In the focusing case, we follow the same idea as Kenig and Merle used to eliminate the critical element for the equation (CP0) in the paper \cite{kenig}. 
\end{itemize}

\subsection{The Defocusing Case: A Morawetz-type Estimate}

In this subsection we introduce the following Morawetz-type estimate and use it to ``kill'' the critical element. 

\begin{proposition} [A  Morawetz-type Inequality] \label{Morawetz1}
Assume $\phi \in C^1 (\Rm^d)$ satisfies the condition \eqref{basic condition of phi} and 
\[
 \eta (x) \doteq \phi (x) - \frac{(d-2) x\cdot \nabla \phi (x)}{2(d-1)} > 0,\qquad x \in \Rm^d 
\]
Let $u$ be a solution to the Cauchy problem (CP1) in the defocusing case with initial data $(u_0,u_1) \in \dot{H}^1 \times L^2 (\Rm^d)$ and a maximal lifespan $(-T_-,T_+)$. Then we have 
\[
 \int_{-T_-}^{T_+} \int_{\Rm^d} \eta (x) \frac{|u|^{2^\ast}}{|x|} dx dt \leq \frac{2d}{d-1} E_\phi (u_0,u_1).
\]
\end{proposition}

\paragraph{The idea of proof} The main idea is to choose a suitable function $a(x)$ and then apply the informal computation 
\begin{align*}
   & -\frac{d}{dt} \int_{\Rm^d} u_t  \left( \nabla a \cdot \nabla u + u \cdot \frac{\Delta a}{2} \right) dx \\
 = & \int_{\Rm^d} \left[\nabla u \cdot \Db^2 a \cdot (\nabla u)^T \right] dx
 - \frac{1}{4} \int_{\Rm^d} \left( |u|^2 \Delta \Delta a\right) dx
 + \int_{\Rm^d} \left(\frac{1}{d}\phi \Delta a - \frac{1}{2^\ast} \nabla \phi \cdot \nabla a \right) |u|^{2^\ast} dx,
\end{align*}
for a solution $u$ to (CP1) in the defocusing case. In order to obtain a Morawetz inequality, the same idea has been used in \cite{benoit} for the defocusing wave equation with a pure power-type nonlinearity and in \cite{shen2} for the defocusing shifted wave equation on the hyperbolic spaces. Here we choose the function $a(x) = |x| = r$, which satisfies
\begin{align*}
 &\nabla a = \frac{x}{r};& &\Delta a = \frac{d-1}{r};& &\Db^2 a \geq 0;& &\nabla \Delta a = - \frac{(d-1)x}{r^3};& &\Delta \Delta a \leq 0.&
\end{align*}
Since the original solution does not possess sufficient smoothness, we need to apply some smoothing and cut-off techniques. Please see \cite{benoit, shen2} for more details on this argument. 

\paragraph{Nonexistence of a critical element} Applying Proposition \ref{Morawetz1}, we obtain a global integral estimate 
\[
 \int_{-\infty}^{\infty} \int_{\Rm^d} \eta (x) \frac{|u|^{2^\ast}}{|x|} dx dt \leq \frac{2d}{d-1} E_\phi (u_0,u_1) < \infty
\]
for the critical element $u$. However, we know that this integral should have been infinite by the almost periodicity, as shown in the lemma below. This gives us a contradiction and finishes the proof in the defocusing case.   

\begin{lemma} \label{H1 lower bound distribution}
Assume that $\beta (x)$ is a positive measurable function defined on $\Rm^d$. Let $u$ be a critical element of (CP1), either in the focusing or defocusing case, as given in Proposition \ref{compactness process} or Proposition \ref{compactness process defocusing}. Then for any given $\tau > 0$, there exists a positive constant $\delta_1$, such that the following inequalities hold for any $t_0 \in \Rm$. 
\begin{align*}
 &\int_{t_0}^{t_0 + \tau} \!\!\int_{\Rm^d}|\nabla u(x,t)|^2 dx dt \geq \delta_1,& &\int_{t_0}^{t_0 + \tau} \!\!\int_{\Rm^d} \beta(x) |u|^{2^\ast} dx dt \geq \delta_1.& 
\end{align*}
\end{lemma}
\begin{proof}
If the lemma were false for some $\tau >0$, then there would exist a sequence of time $\{t_n\}_{n \in {\mathbb Z}^+}$ such that 
\begin{align} \label{contra assum 1}
 &\int_{t_n}^{t_n + \tau} \!\!\int_{\Rm^d} |\nabla u(x,t)|^2 dx dt < 2^{-n}& &\hbox{or}& &\int_{t_n}^{t_n + \tau} \!\!\int_{\Rm^d} \beta(x) |u(x,t)|^{2^\ast} dx dt < 2^{-n}.&
\end{align}
By the almost periodicity we know the sequence $\{(u(\cdot, t_n), u_t (\cdot, t_n))\}_{n \in {\mathbb Z}^+}$ converges to some pair $(v_0,v_1)$ strongly in the space $\dot{H}^1 \times L^2 (\Rm^d)$. Let $v$ be the corresponding solution to (CP1) with initial data $(v_0,v_1)$ and $\tau_1 \in (0,\tau]$ be a small time contained in the lifespan of $v$. Thus we have $\|v\|_{Y([0,\tau_1])} < \infty$. Applying the long-time perturbation theory on the solution $v$, the time interval $[0,\tau_1]$ and the initial data $(u(\cdot, t_n), u_t (\cdot, t_n))$, we obtain that 
\begin{align*}
& \lim_{n \rightarrow \infty} \sup_{t \in [0,\tau_1]} \|u(\cdot, t_n +t) - v(\cdot, t)\|_{\dot{H}^1 (\Rm^d)} = 0,&
& \lim_{n \rightarrow \infty} \|u(\cdot, \cdot + t_n) - v \|_{Y^\star ([0,\tau_1])} = 0.&
\end{align*}
Combining this with our assumption \eqref{contra assum 1}, we obtain 
\begin{align*}
 &\int_{0}^{ \tau_1} \int_{\Rm^d} |\nabla v(x,t)|^2 dx dt = 0& &\hbox{or}& &\int_{0}^{\tau_1} \!\!\int_{\Rm^d} \beta(x) |v(x,t)|^{2^\ast} dx dt = 0.&
\end{align*}
In either case we have $v(x,t) = 0$ for $t \in  [0,\tau_1]$. As a result we have $(v_0,v_1) = 0$, in other words, $\|(u(\cdot, t_n), u_t (\cdot, t_n))\|_{\dot{H}^1 \times L^2} \rightarrow 0$. Thanks to Proposition \ref{scattering with small data}, this immediately gives a contradiction.
\end{proof}

\subsection{The Focusing Case}

The idea is to show the derivative 
\[
 \frac{d}{dt} \left[\int_{\Rm^d}(x \cdot \nabla u) u_t \varphi_R dx + \frac{d}{2} \int_{\Rm^d} u u_t \varphi_R dx\right]
\]
has a negative upper bound but the integral itself is bounded, which gives a contradiction when we consider a long time interval. Here $\varphi_R$ is a cut-off function defined below and the parameter $R$ is to be determined. It is necessary to apply the cut-off techniques here since the functions $(x \cdot \nabla u) u_t$ and $u u_t$ may not be integrable in the whole space.  

\begin{definition}[Cut-off function] \label{definition of varphi}
Let us fix a radial, smooth, nonnegative cut-off function $\varphi: \Rm^d \rightarrow [0,1]$ satisfying 
\[
 \varphi(x) = \left\{\begin{array}{ll} 1, & \hbox{if}\; |x|< 1;\\
 0, & \hbox{if}\; |x|>2;
 \end{array} \right. 
\]
and define its rescaled version 
\[
 \varphi_R (x) = \varphi(x/R).
\]
\end{definition}

\begin{definition}
 If $R > 0$, then we define
\[
 \kappa (R)=  \sup_{t \in \Rm} \int_{|x| > R} \left(|u_t (x, t)|^2 + |\nabla u (x, t)|^2 + \frac{|u(x,t)|^2}{|x|^2} + |u(x,t)|^{2^\ast} \right) dx
\]
\end{definition}
\begin{lemma} Let $u$ be a critical element as in Proposition \ref{compactness process}. Then $\kappa (R)$ is bounded and converges to zero as $R \rightarrow \infty$. 
\end{lemma}
\begin{proof}
This is a direct corollary of the pre-compactness of $\{(u(\cdot, t), \partial_t u(\cdot, t))| t\in \Rm\}$ and the Hardy Inequality. 
\end{proof}
\begin{lemma}[Hardy Inequality] \label{xi2 L2}
Assume $f \in \dot{H}^1 (\Rm^d)$. We have 
\[
 \left(\int_{\Rm^d} \frac{|f(x)|^2}{|x|^2} dx\right)^{1/2} \lesssim \|f\|_{\dot{H}^1 (\Rm^d)}. 
\]
\end{lemma}

\begin{lemma}[Calculation of Derivatives] \label{derivatives of combo}
Fix $R > 0$ and let $\varphi_R$ be the cut-off function as in Definition \ref{definition of varphi}. We have the following derivatives in $t$ 
\begin{align*}
 \frac{d}{dt} \left[\int_{\Rm^d}(x \cdot \nabla u) u_t \varphi_R dx \right]  = & - \frac{d}{2} \int_{\Rm^d} |u_t|^2 dx + \frac{d-2}{2} \int_{\Rm^d} \left(|\nabla u|^2 - \phi |u|^{2^\ast} \right) dx\\
 &\quad - \frac{1}{2^\ast} \int (\nabla \phi \cdot x) |u|^{2^\ast} \varphi_R\, dx + O(\kappa(R));\\
  \frac{d}{dt} \left[\int_{\Rm^d} \varphi_R u u_t dx \right] = & \int_{\Rm^d} |u_t|^2 dx - \int_{\Rm^d} \left(|\nabla u|^2 - \phi |u|^{2^\ast} \right) dx + O(\kappa(R)).
\end{align*}
Here $O(\kappa (R))$ represents an error term that can be dominated by a constant multiple of $\kappa (R)$. 
\end{lemma}
\begin{proof}
We can always work as though $u$ is a smooth solution to (CP1) by smoothing techniques. The idea is to apply integration by parts
\begin{align}
 \frac{d}{dt} \left[\int_{\Rm^d} (x \cdot \nabla u) u_t \varphi_R dx \right] = & \int_{\Rm^d} (x \cdot \nabla u_t) u_t \varphi_R dx + \int_{\Rm^d} (x \cdot \nabla u) u_{tt} \varphi_R dx\nonumber \\
 = & \frac{1}{2} \int_{\Rm^d} \varphi_R x \cdot \nabla (|u_t|^2) dx + \int_{\Rm^d} (x \cdot \nabla u) (\Delta u + \phi |u|^{4/(d-2)}u) \varphi_R dx\nonumber \\
 = & - \frac{d}{2} \int_{\Rm^d} \varphi_R |u_t|^2 dx - \frac{1}{2} \int_{\Rm^d} (\nabla \varphi_R\cdot x) |u_t|^2 dx\nonumber \\
 & \quad - \int_{\Rm^d} \nabla (\varphi_R x \cdot \nabla u)\cdot \nabla u dx + \frac{1}{2^\ast} \int_{\Rm^d} \varphi_R \phi x \cdot \nabla (|u|^{2^\ast}) dx\nonumber \\
 = & - \frac{d}{2} \int_{\Rm^d} |u_t|^2 dx + O(\kappa(R)) + I_1 + I_2 \label{derivative 101}
\end{align}
In the calculation below, we let $u_i$, $u_{ij}$ represent the derivatives $\displaystyle \frac{\partial u}{\partial x_i}$, $\displaystyle \frac{\partial^2 u}{\partial x_j \partial x_i}$, respectively. We have 
\begin{align*}
 I_1 = & - \int_{\Rm^d} \nabla (\varphi_R x \cdot \nabla u)\cdot \nabla u dx\\
 = & - \sum_{i,j=1}^d \int_{\Rm^d} \left(\frac{\partial \varphi_R}{\partial x_j} x_i u_i u_j + \varphi_R \delta_{ij} u_i u_j + \varphi_R x_i u_{ij} u_j\right) dx\\
 = & - \int_{\Rm^d} \varphi_R |\nabla u|^2 dx - \frac{1}{2} \int_{\Rm^d} \varphi_R x \cdot \nabla (|\nabla u|^2) dx  + O(\kappa(R))\\
 = & - \int_{\Rm^d} |\nabla u|^2 dx + \frac{d}{2} \int_{\Rm^d} \varphi_R |\nabla u|^2 dx + \frac{1}{2} \int_{\Rm^d} (x \cdot \nabla \varphi_R) |\nabla u|^2 dx + O(\kappa(R))\\
 = & \frac{d-2}{2} \int_{\Rm^d} |\nabla u|^2 dx + O(\kappa(R))
\end{align*}
\begin{align*}
 I_2 = & \frac{1}{2^\ast} \int_{\Rm^d} \varphi_R \phi x \cdot \nabla (|u|^{2^\ast}) dx\\
 = & - \frac{d}{2^\ast} \int_{\Rm^d} \varphi_R \phi |u|^{2^\ast} dx - \frac{1}{2^\ast} \int_{\Rm^d} (\nabla \phi \cdot x) \varphi_R |u|^{2^\ast} dx - \frac{1}{2^\ast} \int_{\Rm^d} (\nabla \varphi_R \cdot x) \phi |u|^{2^\ast} dx\\
 = & -\frac{d-2}{2} \int_{\Rm^d} \phi |u|^{2^\ast} dx - \frac{1}{2^\ast} \int_{\Rm^d} (\nabla \phi \cdot x) \varphi_R |u|^{2^\ast} dx + O(\kappa(R)).
\end{align*}
Plugging $I_1$, $I_2$ into \eqref{derivative 101}, we finish the calculation of the first derivative. The second derivative can be dealt with in the same manner:
\begin{align*}
 \frac{d}{dt} \left[\int_{\Rm^d} \varphi_R u u_t dx \right] = &\int_{\Rm^d} \varphi_R |u_t|^2 dx + \int_{\Rm^d} \varphi_R u u_{tt} dx\\
 = & \int_{\Rm^d} |u_t|^2 dx + \int_{\Rm^d} \varphi_R u (\Delta u + \phi |u|^{4/(d-2)}u)  dx + O(\kappa(R))\\
 = & \int_{\Rm^d} |u_t|^2 dx - \int_{\Rm^d} \varphi_R |\nabla u|^2 dx + \int_{\Rm^d} \varphi_R \phi |u|^{2^\ast}dx + O(\kappa(R))\\
 = & \int_{\Rm^d} |u_t|^2 dx - \int_{\Rm^d} \left( |\nabla u|^2 - \phi |u|^{2^\ast}\right) dx + O(\kappa(R)).
\end{align*}
\end{proof}

\paragraph{Nonexistence of a critical element} Now we can show that a critical element does not exist in the focusing case. Consider the function ($R>0$)
\[
 G_{R} (t) = \int_{\Rm^d}(x \cdot \nabla u (x,t)) u_t (x,t) \varphi_R dx + \frac{d}{2} \int_{\Rm^d} \varphi_R u(x,t) u_t(x,t) dx.
\]
Applying Lemma \ref{derivatives of combo}, we obtain 
\begin{align}
 G'_{R} (t) = &  - \int_{\Rm^d} \left(|\nabla u|^2 - \phi |u|^{2^\ast} \right) dx - \frac{1}{2^\ast} \int (\nabla \phi \cdot x) |u|^{2^\ast} \varphi_R\, dx + O(\kappa(R))\nonumber\\
 = & -  \int_{\Rm^d} \left(|\nabla u|^2 -  |u|^{2^\ast} \right) dx -  \int_{\Rm^d} (1-\phi) |u|^{2^\ast} \varphi_R dx - \frac{1}{2^\ast} \int (\nabla \phi \cdot x) |u|^{2^\ast} \varphi_R\, dx + O(\kappa(R))\nonumber\\
 \leq & - C(E) \int_{\Rm^d} |\nabla u|^2 dx - \frac{1}{2^\ast}  \int_{\Rm^d} [2^\ast (1-\phi) + (\nabla \phi \cdot x) ] |u|^{2^\ast} \varphi_R\, dx + O(\kappa(R)) \label{final derivative}
\end{align}
In the last step above, we apply Proposition \ref{energy trapping}. The positive constant $C(E)$ only depends on the energy $E = E_\phi (u, u_t)$, but not on $t$ or $R$. According to our assumption that $2^\ast (1 - \phi(x)) + (x\cdot \nabla \phi(x)) \geq 0$ holds for all $x \in \Rm^d$, the integrand of the second integral in \eqref{final derivative} is always nonnegative. Thus we have 
\[
 G'_{R} (t) \leq - C(E) \int_{\Rm^d} |\nabla u|^2 dx  + O(\kappa(R))
\]
for all $R > 0$ and $t \in \Rm$. Fix $\tau >0$ and let $\delta_1$ be the constant in Lemma \ref{H1 lower bound distribution}. Since $\displaystyle \lim_{R\rightarrow \infty} \kappa(R) = 0$, we can fix a large $R$ so that 
\[
 G'_{R} (t) \leq -  C(E)  \int_{\Rm^d} |\nabla u|^2 dx + \frac{C(E) \delta_1}{2\tau}, \qquad \hbox{for any}\; t \in \Rm.
\]
Integrating both sides from $t=0$ to $t = n\tau$ for an positive integer $n$ and applying Lemma \ref{H1 lower bound distribution}, we obtain 
\[
 G_R (n \tau) - G_R (0) \leq - C(E) \int_0^{n\tau}\!\! \int_{\Rm^d} |\nabla u|^2 dx dt + \frac{C(E) \delta_1}{2\tau} \cdot n\tau \leq - \frac{C(E)\delta_1 \cdot n}{2}.
\]
This is impossible as $n \rightarrow \infty$, because $|G_{R}(t)|$ has a uniform upper bound for all $t$ if we fix $R$. 
\begin{align*}
 |G_{R}(t)| \lesssim & R \int_{\Rm^d} \left(|\nabla u(x,t)|^2 + |u_t (x,t)|^2\right) dx +  \int_{|x|<2R} \left(|u(x,t)|^{2^\ast} + |u_t (x,t)|^2 + 1\right) dx\\
 \lesssim & R E + E + R^d.
\end{align*} 
Here we need to use Remark \ref{positive energy}.

\section{Finite Time Blow-up}

In this section, we prove the second part of my main theorem. Namely, if $\|\nabla u_0\|_{L^2} > \|\nabla W\|_{L^2}$ and $E_\phi (u_0,u_1) < E_1 (W,0)$, then the corresponding solution to (CP1) in the focusing case blows up within finite time in both two time directions. Since our argument here is similar to the one used in section 7 of \cite{kenig}, we will omit some details. 

\paragraph{The idea} If the initial data $u_0 \in L^2 (\Rm^d)$, then the blow-up of the solution $u$ can be proved by considering the function $y (t) = \int_{\Rm^d}|u(x,t)|^2 dx$ and showing that this function has to blow up in finite time. In the general case, we have to use a cut-off technique. Let us assume $T_+ (u_0,u_1) = + \infty$ and show a contradiction. Applying integration by parts and smoothing approximation techniques, we have 

\begin{lemma} \label{derivatives of comb 2}
 Let $\varphi_R (x)$ be the cut-off function as given in Definition \ref{definition of varphi}. If we define $\displaystyle y_R (t) = \int_{\Rm^d} |u(x,t)|^2 \varphi_R(x) dx$, then 
\begin{align*}
 y'_R (t) & = 2 \int_{\Rm^d} u_t (x,t) u(x,t) \varphi_R (x) dx\\
 y''_R (t) & = - \frac{4d}{d-2} \int_{\Rm^d} \left(\frac{1}{2} |\nabla u|^2 + \frac{1}{2} |u_t|^2 - \frac{1}{2^\ast} \phi |u|^{2^\ast}\right)\varphi_R dx + \frac{4(d-1)}{d-2} \int_{\Rm^d} |u_t|^2 \varphi_R dx\\
 & \qquad + \frac{4}{d-2} \int_{\Rm^d} |\nabla u|^2 \varphi_R dx -2 \int_{\Rm^d} (\nabla u \cdot \nabla \varphi_R) u dx. 
\end{align*}
\end{lemma}

\paragraph{Tail Estimate} In order to dominate the error of the integral created by the cut-off, we need to make a ``tail'' estimate. First of all, for any given initial data $(u_0,u_1) \in \dot{H}^1 \times L^2 (\Rm^d)$ we always have 
\[
 \lim_{R \rightarrow \infty} \left\|\left((1-\varphi_R (x))u_0, (1-\varphi_R (x))u_1\right)\right\|_{\dot{H}^1 \times L^2 (\Rm^d)} = 0.
\]
Thus for any $\eps > 0$, there exists a number $R_0 = R_0 (\eps)$, such that 
\[
  \left\|\left((1-\varphi_{R_0} (x))u_0, (1-\varphi_{R_0} (x))u_1\right)\right\|_{\dot{H}^1 \times L^2 (\Rm^d)} < \eps.
\]
When $\eps$ is sufficiently small, our local theory guarantees that the solution $u_{R_0}$ to (CP1) with initial data $\left((1-\varphi_{R_0} (x))u_0, (1-\varphi_{R_0} (x))u_1\right)$ exists globally in time and scatters with 
\[
 \sup_{t \in \Rm} \|(u_{R_0}(\cdot, t), \partial_t u_{R_0}(\cdot ,t))\|_{\dot{H}^1 \times L^2 (\Rm^d)} < 2\eps. 
 \]
As a result, we have the following estimate for each $t \in R$:
\[
 \int_{\Rm^d} \left( |\nabla u_{R_0}(x,t)|^2 + |\partial_t u_{R_0}(x,t)|^2 + |u_{R_0}(x,t)|^{2^\ast} + \frac{|u_{R_0}(x,t)|^2}{|x|^2}\right) dx \leq C \eps^2. 
\]
Here $C > 1$ is a constant depending only on the dimension $d$. Since our cut-off version of initial data $\left((1-\varphi_{R_0} (x))u_0, (1-\varphi_{R_0} (x))u_1\right)$ remain the same as the original initial data $(u_0,u_1)$ in the region $\{x: |x| \geq 2 R_0\}$, finite speed of propagation immediately gives 
\begin{equation} \label{tail estimate}
  \int_{|x|> 2R_0 (\eps) + |t| } \left( |\nabla u (x,t)|^2 + |\partial_t u (x,t)|^2 + |u (x,t)|^{2^\ast} + \frac{|u (x,t)|^2}{|x|^2}\right) dx \leq C \eps^2
\end{equation}

\paragraph{Proof of the blow-up part} If $R > 2 R_0 (\eps)$ and $t \in [0, R-2R_0 (\eps)]$, then we can combine Lemma \ref{derivatives of comb 2} with the tail estimate \eqref{tail estimate} and obtain 
\[
 y''_R (t) = -\frac{4d}{d-2} E_\phi (u,u_t) + \frac{4(d-1)}{d-2} \int_{\Rm^d} |u_t|^2 \varphi_R dx + \frac{4}{d-2} \int_{\Rm^d} |\nabla u|^2 dx + O(\eps^2). 
\]
We have already known $E_\phi (u,u_t) = E_\phi (u_0,u_1) < E_1 (W,0)$. In addition, we claim that the inequality $\|\nabla u(\cdot, t)\|_{L^2} \geq \|\nabla W\|_{L^2}$ holds for all $t\geq 0$. Otherwise we would have $\|\nabla u_0\|_{L^2} < \|\nabla W\|_{L^2}$ by Lemma \ref{energy trapping}. Using these inequalities and the identity $\|\nabla W\|_{L^2}^2 = d \cdot E_1(W,0)$ we can find a lower bound of the second derivative 
\begin{align*}
 y''_R(t) \geq & \frac{4d}{d-2} \left(E_1(W,0) - E_\phi (u_0,u_1)\right) + \frac{4(d-1)}{d-2} \int_{\Rm^d} |u_t|^2 \varphi_R dx\\
  & \qquad + \frac{4}{d-2} \left(\int_{\Rm^d} |\nabla W|^2 dx - d\cdot  E_1 (W,0)\right) + O(\eps^2)\\
  \geq & \delta + \frac{4(d-1)}{d-2} \int_{\Rm^d} |u_t|^2 \varphi_R dx  - C_1 \eps^2.
\end{align*}
Here the constant $C_1$ is determined solely by the dimension $d$ while $\delta$ can be arbitrarily chosen in the interval $\left(0, \frac{4d}{d-2} \left(E_1(W,0) - E_\phi (u_0,u_1)\right)\right)$. Let us fix $\delta \ll 1$ and $\eps = \delta^2$ so that 
\begin{align} \label{choice of delta}
&\delta - C_1 \eps^2 = \delta -C_1 \delta^4 > \delta/2,& &\delta < \min\left\{C, \frac{1}{100C}, \frac{4d}{d-2} \left(E_1(W,0) - E_\phi (u_0,u_1)\right)\right\}&
\end{align}
Here $C$ is the constant in the inequality \eqref{tail estimate}. As a result we have if $R > 2 R_0 (\delta^2)$ and $t \in [0, R - 2 R_0 (\delta^2)]$, then 
\begin{equation} \label{estimate on the second derivative}
 y''_R (t) \geq \frac{\delta}{2} + \frac{4(d-1)}{d-2} \int_{\Rm^d} |u_t|^2 \varphi_R dx;\quad \Longrightarrow \quad y''_R(t) y_R(t) \geq \frac{d-1}{d-2} [y'_R (t)]^2. 
\end{equation} 
In addition, we have the following estimates on $y_R (0)$ and $|y'_R (0)|$. In the integrals below $\Omega$ represents the region $\{x: 2R_0(\delta^2) < |x| < 2R\}$. 
\begin{align*}
 y_R (0) \leq & \int_{|x|\leq 2 R_0(\delta^2)} |u_0|^2 dx + 4R^2 \int_{\Omega}  \frac{|u_0|^2}{|x|^2} dx \leq \int_{|x|\leq 2 R_0(\delta^2)} |u_0|^2 dx + 4C \delta^4 R^2; \\
 |y'_R (0)| \leq & 2 \int_{|x|\leq 2R_0(\delta^2)} |u_0| \cdot |u_1|\, dx + 2 \left(\int_{\Omega} |u_1|^2 dx \right)^{1/2} \left(4R^2 \int_{\Omega} \frac{|u_0|^2}{|x|^2} dx\right)^{1/2}\\
 \leq & 2 \int_{|x|\leq 2R_0(\delta^2)} |u_0| \cdot |u_1|\, dx + 4C\delta^4 R.
\end{align*}
As a result, we always have the following estimates for sufficiently large $R > R_1$:
\begin{align} \label{choice of R 1}
 &y_R (0) < 5C \delta^4 R^2;& &|y'_R (0)| < 5C \delta^4 R;& &R > 100 R_0 (\delta^2).& 
\end{align}
Let us consider $t_0 (R) \doteq \min \{t: 0 \leq t \leq 12 C \delta R,\; y'_R (t) \geq C\delta^2 R\}$ for such a radius $R$. This is well-defined since we know $12 C\delta R < 12R/100 < R- 2 R_0 (\delta^2)$ and 
\[
 y'_R (12 C \delta R) \geq y'_R (0) + 12 C \delta R \cdot \inf_{0\leq t\leq 12C\delta R} y''_R (t) \geq -5C \delta^4 R + 12 C \delta R \cdot \frac{\delta}{2} \geq C \delta^2 R. 
\] 
In addition we have 
\begin{equation} \label{upper bound of yR}
 y_R (t_0(R)) \leq y_R (0) + t_0(R) \cdot \max_{0 \leq t \leq t_0(R)} y'_R (t) \leq 5C \delta^4 R^2 + 12 C \delta R \cdot C \delta^2 R \leq 17 C^2 \delta^3 R^2.
\end{equation}
Now let us define $\displaystyle z_R (t) = \frac{y'_R (t)}{y_R (t)}$ for $t \in [t_0 (R), R - 2 R_0 (\delta^2)]$. The function $z_R (t)$ is always positive. By the estimate \eqref{upper bound of yR} and the definition of $t_0 (R)$, we have $\displaystyle z(t_0(R)) \geq \frac{C \delta^2 R}{17 C^2 \delta^3 R^2} = \frac{1}{17 C \delta R}$. Combining basic differentiation and the estimate \eqref{estimate on the second derivative}, we have 
\[
 z'_R (t) = \frac{y''_R (t) y_R (t) - [y'_R (t)]^2}{y_R^2 (t)} \geq \frac{1}{d-2} \left[\frac{y'_R (t)}{y_R (t)}\right]^2 = \frac{1}{d-2} z_R^2 (t).
\]
for any $t \in [t_0(R), R - 2 R_0 (\delta^2)]$. Dividing both sides by $z_R^2 (t)$ and integrating in $t$, we obtain 
\[
 \frac{1}{z(t_0(R))} - \frac{1}{z(R - 2 R_0 (\delta^2))} \geq \frac{1}{d-2} \left[(R - 2 R_0 (\delta^2)) - t_0(R) \right]
\]
Using the upper bound of $t_0 (R)$, the lower bound of $z(t_0(R))$ and the choice of $R$, we have   
\[
 17 C \delta R \geq \frac{1}{z(t_0(R))} > \frac{1}{d-2}\left[R - \frac{R}{50} - 12 C \delta R\right] \quad \Longrightarrow \quad [17(d-2) + 12] C \delta R > \frac{49}{50} R.
\]
This contradicts our choice of $\delta$, please see \eqref{choice of delta}. 

\section{Application on a Shifted Wave Equation on $\Hm^3$}

In this section we consider the radial solutions to an energy-critical, focusing, semilinear shifted wave equation on the hyperbolic space $\Hm^3$
\begin{equation} \label{application equation}
 \left\{\begin{array}{ll} 
         \partial_t^2 v - (\Delta_{\Hm^3} +1) v = |v|^4 v, & (y,t) \in \Hm^3 \times \Rm;\\
         v(\cdot,0) = u_0 \in H^{0,1} (\Hm^3);\\
         \partial_t v(\cdot,0) = u_1 \in L^2 (\Hm^3);
        \end{array} \right.
\end{equation}
as one application of our main theorem. 

\subsection{Background and the Space of Functions}

\paragraph{Model of hyperbolic space} There are various models for the Hyperbolic space $\Hm^3$. We select to use the hyperboloid model. Let us consider the Minkowswi space $\Rm^{3+1}$ equipped with the standard Minkowswi metric $-(dx^0)^2 + (dx^1)^2 + \cdots + (dx^3)^2$ and the bilinear form $[x,y] = x_0 y_0 - x_1 y_1 - x_2 y_2 - x_3 y_3$. The hyperbolic space $\Hm^3$ can be defined as the hyperboloid $x_0^2 - x_1^2 - x_2^2 -x_3^2 =1$ whose metric, covariant derivatives and measure are induced by the Minkowswi metric.

\paragraph{Radial Functions} We can introduce polar coordinates $(r, \Theta)$ on the hyperbolic space $\Hm^3$. More precisely, we use the pair $(r, \Theta) \in [0,\infty) \times {\mathbb S}^{2}$ to represent the point $(\cosh r, \Theta \sinh r) \in \Rm^{3+1}$ in the hyperboloid model above. One can check that the $r$ coordinate of a point in $\Hm^3$ represents the distance from that point to the ``origin'' $\mathbf{0} \in \Hm^n$, which is the point $(1,0,0,0)$ in the Minkiwski space. In terms of the polar coordinate, the measure and Laplace operator can be given by 
\begin{align*}
 &d\mu = \sinh^2 r \,dr d\Theta;& &\Delta_{\Hm^3} = \partial_r^2 + 2 \coth r \cdot \partial_r + \sinh^{-2} r \cdot \Delta_{{\mathbb S}^2}& 
\end{align*}
Here $d\Theta$ corresponds the usual unnormalized measure on the sphere ${\mathbb S}^2$. As in Euclidean spaces, for any $y\in \Hm^3$ we also use the notation $|y|$ for the distance from $y$ to $\mathbf{0}$. Namely
\[
 r = |y| = d (y, \mathbf{0}), \qquad y \in \Hm^n.
\]
A function $f$ defined on $\Hm^3$ is radial if it is independent of $\Theta$. By convention we can use the notation $f(r)$ to mention a radial function $f$.

\paragraph{Function Space} The homogenous Sobolev space $H^{0,1}(\Hm^3)$, which is the counterpart of $\dot{H}^1 (\Rm^d)$ in the hyperbolic space $\Hm^3$, is defined by 
\begin{align*}
 &H^{0,1} (\Hm^3) = (-\Delta_{\Hm^3} -1)^{-1/2} L^2 (\Hm^3) & &\|u\|_{H^{0,1} (\Rm^3)} = \|(-\Delta_{\Hm^3} -1)^{1/2} u\|_{L^2 (\Hm^3)}&
\end{align*}
If $f \in C_0^\infty (\Hm^3)$, then its $H^{0,1} (\Hm^3)$ norm can also be given by ($|\nabla f| = (\Db_\alpha f \Db^\alpha f)^{1/2}$)
\begin{equation} \label{alternative definition of H01}
 \|f\|_{H^{0,1} (\Rm^3)} = \int_{\Hm^3} \left(|\nabla f (y)|^2 - |f(y)|^2\right) d\mu(y).
\end{equation}
Please pay attention that the spectrum of the Laplace operator $- \Delta_{\Hm^3}$ is $[1,\infty)$, which is much different from that of the Laplace operator on $\Rm^d$. As a result, the integral above is always nonnegative. 

\paragraph{Sobolev Embedding} As in Euclidean Spaces, we have the Sobolev embedding $H^{0,1} (\Hm^3) \hookrightarrow L^6 (\Hm^3)$. (Please see \cite{subhyper} for more details.) This implies that the energy 
\begin{equation} \label{energy of v}
 E(v, \partial_t v) = \frac{1}{2}\|v\|_{\dot{H}^{0,1}(\Hm^3)}^2 + \frac{1}{2} \|\partial_t v\|_{L^2 (\Hm^3)}^2 - \frac{1}{6} \|v\|_{L^6 (\Hm^3)}^6 = E(v_0,v_1)
\end{equation}
is a finite constant as long as $(v_0,v_1) \in \dot{H}^{0,1}\times L^2 (\Hm^3)$.

\paragraph{Local Theory} Both the Strichartz estimates and local theory have been discussed in my recent work \cite{echyper}. Generally speaking, the local theory is similar to that of a wave equation on the Euclidean space. Given any initial data $(v_0,v_1) \in \dot{H}^{0,1}\times L^2 (\Hm^3)$, there is a unique solution $v$ defined on a maximal interval of time $I$, such that $(v, \partial_t v) \in C(I; H^{0,1}\times L^2 (\Hm^3))$ and the inequality $\|v\|_{L^5 L^{10} (J \times \Hm^3)} < \infty$ holds for any bounded closed subinterval $J$ of $I$.   

\subsection{A transformation}

Let us consider the transformation ${\mathbf T} : L^2 (\Hm^3) \rightarrow L^2 (\Rm^3)$ defined by
\[
 ({\mathbf T} f) (r, \Theta) = \frac{\sinh r}{r} f (r, \Theta).
\]
Here $(r, \Theta) \in [0,\infty) \times {\mathbb S}^2$ represents the polar coordinates, in either the hyperbolic space $\Hm^3$ or the Euclidean space $\Rm^3$. It is trivial to check that this transformation is actually an isometry from $L^2 (\Hm^3)$ to $L^2 (\Rm^3)$. In addition, the transformation $\mathbf T$ is also an isometry from $H_{rad}^{0,1} (\Hm^3)$ to $\dot{H}_{rad}^{1} (\Rm^3)$, where two spaces of radial functions are subspaces defined by 
\begin{align*}
 &H_{rad}^{0,1} (\Hm^3) = \{f \in H^{0,1}(\Hm^3): f\; \hbox{is radial}\}& &\dot{H}_{rad}^{1} (\Rm^3) = \{f \in \dot{H}^1 (\Rm^3): f\; \hbox{is radial}\}.&
\end{align*}
This can be observed by using the identity \eqref{alternative definition of H01} and conducting basic calculations. Furthermore, if $f$ is a radial and smooth function defined on $\Hm^3$, then one can also verify 
\[
 - \Delta_{\Rm^3} \left({\mathbf T} f\right) = {\mathbf T} \left[(-\Delta_{\Hm^3} -1) f\right].
\]
Combining all the facts above, we have 

\begin{lemma} \label{transformation between two equations}
 If $v(y,t)$ is a solution to the equation \eqref{application equation} with initial data $(u_0,u_1) \in H_{rad}^{0,1} \times L_{rad}^2 (\Hm^3)$ and a maximal lifespan $I$, then $u (\cdot,t) = \mathbf T v(\cdot, t)$ is a solution to the equation 
 \begin{equation} \label{equation for u application}
  \partial_t^2 u - \Delta_{\Rm^3} u = \phi(x) |u|^4 u, \qquad (x,t) \in \Rm^3 \times \Rm
 \end{equation} 
 with the initial data $({\mathbf T} u_0, {\mathbf T} u_1)$ and the same maximal lifespan $I$. Here the coefficient function $\phi$ is defined by 
$\phi(x) = \frac{|x|^4}{\sinh^4 |x|}$. In addition, the energy is also preserved under this transformation. Namely, the energy 
\[
 E_\phi (u,\partial_t u) = \int_{\Rm^3} \left[\frac{1}{2}|\nabla u| + \frac{1}{2} |\partial_t u|^2 - \frac{\phi}{6} |u|^6 \right] dx
\]
remains the same as the energy $E(v,\partial_t v)$ defined in \eqref{energy of v}. 
\end{lemma}

\subsection{Conclusion}

According to Remark \ref{example of phi focusing}, our main theorem may be applied to the equation \eqref{equation for u application}. Combing our main theorem and Lemma \ref{transformation between two equations}, we immediately obtain 

\begin{theorem} \label{hyperbolic theorem focusing}
Given a pair of initial data $(v_0,v_1) \in H_{rad}^{0,1} \times L_{rad}^2 (\Hm^3)$ with an energy $E(v_0,v_1) < E_1 (W,0)$, the global behaviour, and in particular, the maximal lifespan $I = (- T_- (v_0,v_1), T_+ (v_0,v_1))$ of the corresponding solution $v$ to the Cauchy problem \eqref{application equation} can be determined by:
\begin{itemize}
 \item[(I)] If $\|v_0\|_{H^{0,1}(\Hm^3)} < \|\nabla W\|_{L^2 (\Rm^3)}$, then $I = \Rm$ and $v$ scatters in both time directions. 
 \item[(II)] If $\| v_0\|_{H^{0,1} (\Hm^3)} > \|\nabla W\|_{L^2(\Rm^3)}$, then $v$ blows up within finite time in both two directions, namely 
 \begin{align*}
  &T_- (v_0,v_1) < +\infty;& &T_+ (v_0,v_1) < +\infty.&
 \end{align*}
\end{itemize}
\end{theorem}

\end{document}